\documentclass[12pt]{article}

\usepackage{amsmath,amsthm,amsfonts,amssymb, color,xcolor,subcaption,graphicx} 

\usepackage[normalem]{ulem}
\usepackage{todonotes, dsfont}

\newtheorem{theorem}{Theorem}[section]
\newtheorem{corollary}[theorem]{Corollary}
\newtheorem{example}[theorem]{Example}

\newtheorem{lemma}[theorem]{Lemma}
\newtheorem{definition}[theorem]{Definition}
\newtheorem{remark}[theorem]{Remark}

\def\cB{\mathcal{B}}
\def\cC{\mathcal{C}}
\def\cD{\mathcal{D}}
\def\cE{\mathcal{E}}
\def\cF{\mathcal{F}}

\def\cH{\mathcal{H}}

\def\cL{\mathcal{L}}

\def\cN{\mathcal{N}}
\def\cO{\mathcal{O}}
\def\cP{\mathcal{P}}

\def\cS{\mathcal{S}}

\def\cZ{\mathcal{Z}}

\def\bC{\mathbb{C}}

\def\bE{\mathbb{E}}

\def\bN{\mathbb{N}}
\def\bP{\mathbb{P}}
\def\bR{\mathbb{R}}

\newcommand{\fH}{\mathfrak{H}}

\def\e{\varepsilon}
\def\k{\kappa}

\def\bt{{\bf t}}

\newcommand{\fm}{\mathfrak{m}}

\begin{document}

\title{SPDEs with time-independent\\
L\'evy colored noise}

\author{Raluca M. Balan\footnote{Corresponding author. University of Ottawa, Department of Mathematics and Statistics, 150 Louis Pasteur Private, Ottawa, Ontario, K1N 6N5, Canada. E-mail address: rbalan@uottawa.ca.} \footnote{Research supported by a grant from the Natural Sciences and Engineering Research Council of Canada.}
\and
Jinxin Wang \footnote{University of Ottawa, Department of Mathematics and Statistics, 150 Louis Pasteur Private, Ottawa, Ontario, K1N 6N5, Canada. E-mail address: jwang023@uottawa.ca.}
}

\date{April 27, 2026}
\maketitle

\begin{abstract}
\noindent
 In this article, we introduce a time-independent version of the L\'evy colored noise considered in \cite{B15,BJ26}. We study the existence of the solution of a linear stochastic partial differential equation with this type of noise, and we identify some necessary conditions which guarantee that the solution has finite $p$-th order moments. Using tools from Malliavin calculus, we investigate the existence of the solution for the equation with multiplicative noise. As examples, we consider the stochastic heat and wave equations in any dimension $d \geq 1$.
\end{abstract}


\vspace{1mm}


\pagebreak


\section{Introduction}

SPDEs have been studied intensively in the literature in the last three decades, using either the random field approach initiated by Walsh \cite{walsh86} or the semigroup approach of Da Prato and Zabczyk \cite{DZ92}. We mention briefly some important developments and few of classical references, focusing only on the random field approach, without aiming to include a comprehensive list. 

Classical models based on the random field approach involve equations perturbed by a space-time Gaussian noise, a space-time generalization of Brownian motion. In \cite{dalang99}, Dalang introduced a spatially homogeneous Gaussian noise (which is correlated in space), and developed an It\^o integral with respect to this noise. Subsequently, Nualart and his co-authors studied equations with temporally-correlated Gaussian noise, using tools from Malliavin calculus, e.g. \cite{nualart06}. Many interesting properties of the solutions of these equations have been discovered in the recent years, such as intermittency \cite{K14}, or exact asymptotic behavior of moments \cite{chen17}. 

At the same time, models driven by a noise which resembles a L\'evy process have also been studied, but with more limited success. 
This noise was called a {\em L\'evy white noise} \cite{B15}, or a {\em cylindrical L\'evy noise} \cite{kumar-riedle20}, depending on what approach was used. If the noise has finite variance, one can develop an It\^o calculus \cite{applebaum09} and a Malliavin calculus \cite{last16}, using the compensated version of the underlying Poisson random measure of the noise. If the noise has infinite variance, truncation techniques have to be used to eliminate the large jumps of the noise \cite{chong17-SPA,CDH19}. The {\em L\'evy colored noise}
was introduced in \cite{B15} using the convolution with a ``coloration'' kernel, and an extensive study of SPDEs with this type of noise can be found in \cite{BJ26}.

In all the afore-mentioned references, the noise includes a time component. In \cite{hu01}, Hu studied for the first time an SPDE (namely the Parabolic Anderson Model, or PAM), driven by a time-independent Gaussian noise, and discovered results which were slightly different compared with the time-dependent case. Still in the Gaussian case, a parallel analysis of various properties of the solutions (for instance Feynman-Kac-type formulas, intermittency and H\"older continuity) has been carried out in \cite{HHNT15} for (PAM) with  time-dependent and time-independent case. More recently, \cite{BY23} extended to the time-independent  noise, the Quantitative Central Limit theorem (QCLT) obtained in \cite{HNV20} for (PAM) with space-time Gaussian white noise.

In this context, a natural question arises: to what extent, some of the (limited) results that are known for the time-dependent L\'evy noise remain valid in the time-independent case? The major barrier that one faces is the lack of any type of martingale structure, which means that It\^o calculus cannot be used, leaving Malliavin calculus as the only option. Another barrier is the absence of an estimate for the $p$-th moment of a multiple Poisson integral: unlike the Gaussian case, on the Poisson space, hyper-contractivity property does not hold. A ``restricted'' hyper-contractivity property was proved in \cite{NPY19}. Despite these significant challenges, some positive answers to the question above exist, and we gather them here. More precisely, in the present article, we show that most the properties related to moments of the solutions of  linear SPDEs discovered in \cite{BJ26} for the L\'evy colored noise remain valid in the time-independent case, under natural assumptions on the fundamental solution of the differential operator. Moreover, under slightly stronger versions of these assumptions, the equations with multiplicative noise have unique solutions, for which we provide the explicit Poisson-chaos expansion. Finally, we show that these assumptions are satisfied by the fundamental solutions of the heat and wave operators in any dimension, provided that the spectral measure $\mu$ associated with the spatial coloration kernel $\k$ of the noise satisfies Dalang's condition (as in the Gaussian case).

\medskip

We begin now to introduce the precise definition of the noise, and associated objects.
Let $N$ be a Poisson random measure (PRM) on the space $({\bf Z},\cZ)$ of intensity $\fm$, where
\begin{equation}
\label{def-Z}
({\bf Z},\cZ,\fm)=\big(\bR^d \times \bR_0,\ \cB(\bR^d) \otimes \cB(\bR_0), \ {\rm Leb} \times \nu\big),
\end{equation}
defined on a complete probability space $(\Omega,\cF,\bP)$. We refer to \cite{resnick07} for the definition of a PRM. Let $\widehat{N}(F)=N(F)-\fm(F)$ be the compensated version of $N$. The space $\bR_0:=\bR \verb2\2 \{0\}$ is endowed with the distance $d(x,y)=|x^{-1}-y^{-1}|$. We assume that $\nu$ is a measure on $\bR_0$ which satisfies the condition:
\[
m_2<\infty,
\]
where
\[
m_p:=\int_{\bR_0}|z|^p \nu(dz) \quad \mbox{for any $p>0$}.
\]

\medskip

We define the {\em (time-independent) L\'evy white noise} $L$ by:
\begin{equation}
\label{def-L}
L(A)=\int_{A \times \bR_0}z \widehat{N}(dx,dz) \quad \mbox{for all $A \in \cB_b(\bR^d)$},
\end{equation}
where $\cB_b(\bR^d)$ is the class of bounded Borel sets of $\bR^d$. Then, for any set $A \in \cB_b(\bR^d)$, 
\[
\bE[L(A)]=0 \quad \mbox{and} \quad \bE|L(A)|^2=m_2|A|,
\] where $|A|$ is the Lebesgue measure of $A$.

We denote $L(1_A)=L(A)$ for any set $A \in \cB_b(\bR^d)$. The map $1_{A} \mapsto L(1_A) \in L^2(\Omega)$ is an isometry which can be extended to $L^2(\bR^d)$. For any $\varphi \in L^2(\bR^d)$, we let
\[
L(\varphi):=\int_{\bR^d}\varphi(x) L(dx)=\int_{\bR^d \times \bR_0} \varphi(x)z \widehat{N}(dx,dz).
\]
Then
\[
\bE[L(\varphi)L(\psi)]=m_2 \int_{\bR^d}\varphi(x) \psi(x)dx \quad \mbox{for any $\varphi,\psi \in L^2(\bR^d)$}.
\]

For any $\varphi \in L^2(\bR^d)$, $L(\varphi)$ has characteristic function:
\begin{equation}
\label{ch-funct-L}
\bE(e^{i\theta L(\varphi)})=\exp\left\{ \int_{\bR_0} \int_{\bR^d}\big(e^{i\theta \varphi(x) z}-1-i\theta \varphi(x)z\big) dx \nu(dz) \right\} \quad \mbox{for any} \ \theta \in  \bR.
\end{equation}

In particular, letting $L(A)=L(1_{A})$ for any $A \in \cB_b(\bR^d)$, we have:
\[
\bE(e^{i\theta L(A)})=\exp\left\{ |A|\int_{\bR_0} \big(e^{i\theta z}-1-i\theta z\big) \nu(dz) \right\} \quad \mbox{for any} \ \theta \in  \bR.
\]

Next, we introduce the {\em (time-independent) L\'evy colored noise}, given by:
\[
X(\varphi)=L(\varphi*\k) \quad \mbox{for all} \ \varphi \in \cS(\bR^d),
\]
where $\cS(\bR^d)$ is the set of rapidly decreasing functions on $\bR^d$. 
Then $X=\{X(\varphi);\varphi \in \cS(\bR^d)\}$ is the time-independent version of 
the {\em L\'evy colored noise} introduced in \cite{BJ26} as a counterpart of the spatially-homogeneous Gaussian noise considered in \cite{dalang99}.

\medskip

We denote by $x\cdot y=\sum_{j=1}^d x_j y_j$ the inner product between $x=(x_1,\ldots,x_d)\in \bR^d$ and $y=(y_1,\ldots,y_d) \in \bR^d$, and we let $|x |=(x \cdot x)^{1/2}$ be the Euclidean norm of $x$. We let $\cF \varphi(\xi)=\int_{\bR^d}e^{-i\xi \cdot x}\varphi(x)dx$ be the Fourier transform of $\varphi \in L^1(\bR^d)$. 

\medskip

Throughout this article, we assume that $\k:\bR^d \to [0,\infty]$ is a tempered function which satisfies Assumption A1 of \cite{BJ26}. We recall this assumption below: 

\medskip

\noindent {\bf Assumption A1.} 
Let $\k:\bR^d \to [0,\infty]$ be a tempered function such that:

(a) the Fourier transform $h:=\cF \k$ in $\cS_{\bC}'(\bR^d)$ is a tempered function on $\bR^d$:
\[
\int_{\bR^d} h(\xi)\phi(\xi)d\xi=\int_{\bR^d}\k(x) \cF \phi(x)dx \quad \mbox{for all $\phi \in \cS_{\bC}(\bR^d)$}.
\]

(b) $|h|^2$ is a tempered function;

(c) there exists a tempered function $f:\bR^d \to [0,\infty]$ such that $|h|^2=\cF f$ in $\cS'(\bR^d)$:
\begin{equation}
\label{Four-h2}
\int_{\bR^d} |h(\xi)|^2 \phi(\xi)d\xi=\int_{\bR^d}f(x) \cF \phi(x)dx \quad \mbox{for all $\phi \in \cS_{\bC}(\bR^d)$}.
\end{equation}

Then $f=\k*\widetilde{\k}$, where $\widetilde{\k}(x)=\k(-x)$ for all $x \in \bR^d$,
and $\cF f=\mu$, where $\mu$ is the tempered measure on $\bR^d$ given by:

\begin{equation}
\label{def-mu}
\mu(d\xi)=\frac{1}{(2\pi)^d}|\cF \k(\xi)|^2 d\xi.
\end{equation}

For any $\varphi,\psi \in \cS(\bR^d)$, we have:
\[
\int_{\bR^d} \int_{\bR^d}\varphi(x)\psi(y)f(x-y)dxdy=\int_{\bR^d}\cF \varphi(\xi) \overline{\cF \psi(\xi)}\mu(d\xi).
\]

Below are some examples of kernels which satisfy Assumption A1. 

\begin{example} (The heat kernel)
\label{heat-ex}
{\rm Let  $\k=H_{d,\alpha/2}$ for some $\alpha>0$, where
\[
H_{d,\alpha}(x)=\frac{1}{(2\pi \alpha)^{d/2}}\exp\left(-\frac{|x|^2}{2\alpha}\right) \quad \mbox{for $\alpha>0$}.
\]
Assumption A holds with $f=\k * \k=H_{d,\alpha}$ and $|\cF \k(\xi)|^2=\cF f(\xi)=\exp(-\alpha|\xi|^2/2)$.
}
\end{example}

\begin{example} (The Riesz kernel)
\label{Riesz-ex}
{\rm Let $\k=R_{d,\alpha/2}$ for some $\alpha\in (0,d)$, where
\[
R_{d,\alpha}(x)=C_{d,\alpha}|x|^{-(d-\alpha)} \quad \mbox{with} \quad C_{d,\alpha}=\pi^{-d/2}2^{-\alpha}\frac{\Gamma(\frac{d-\alpha}{2})}{\Gamma(\frac{\alpha}{2})}.
\]
It is known that $\cF R_{d,\alpha}(\xi)=|\xi|^{-\alpha}$ and $R_{d,\alpha}*R_{d,\beta}=R_{d,\alpha+\beta}$ (see \cite{stein70}, p.118). Hence, Assumption A holds with $f=\k * \k=R_{d,\alpha}$ and $|\cF \k (\xi)|^2=\cF f(\xi)=|\xi|^{-\alpha}$.
}
\end{example}

\begin{example} (The Bessel kernel)
\label{Bessel-ex}
{\rm Let $\k=B_{d,\alpha/2}$ for some $\alpha>0$, where
\[
B_{d,\alpha}(x)=\frac{1}{\Gamma(\alpha/2)}\int_0^{\infty} w^{\alpha/2-1}e^{-w} \frac{1}{(4\pi w)^{d/2}}\exp\left(-\frac{|x|^2}{4w}\right)dw.
\]
It is known that $\cF B_{d,\alpha}(\xi)=(1+|\xi|^{2})^{-\alpha/2}$ and $B_{d,\alpha}*B_{d,\beta}=B_{d,\alpha+\beta}$ (see \cite{stein70}, p.131).  Hence, Assumption A holds with $f=\k * \k=B_{d,\alpha}$ and $|\cF \k (\xi)|^2=\cF f(\xi)=(1+|\xi|^2)^{-\alpha/2}$. Moreover, $B_{d,\alpha}\in L^1(\bR^d)$.
}
\end{example}

By Lemma 2.7 of \cite{BJ26}, for any $\varphi \in \cS(\bR^d)$, $\varphi* \k \in L^2(\bR^d)$ and $\cF (\varphi * \k)=\cF \varphi \cF \k$ in $L^2_{\bC}(\bR)$. Hence,
the noise $X$ is well-defined. Moreover, for $\varphi,\psi \in \cS(\bR^d)$,
\begin{align}
\nonumber
\bE[X(\varphi) X(\psi)]&=\bE[L(\varphi * \k) L(\psi *\k)]=m_2\langle \varphi*\k,\psi*\k \rangle_{L^2(\bR^d)}\\
\label{def-prod0}
&=m_2 \int_{\bR^d} \int_{\bR^d}\varphi(x)\psi(y)f(x-y)dxdy =:m_2 \langle \varphi,\psi \rangle_0.
\end{align}

The map $\varphi \mapsto X(\varphi) \in L^2(\Omega)$ is an isometry which can be extended to $\cP_{0,d}(\bR^d)$, the completion of $\cD(\bR^d)$ with respect to $\langle \cdot,\cdot \rangle_0$.  Then $\cP_{0,d}(\bR^d)$ is the space of (deterministic) integrands with respect to $X$.  We will use the notation:
\[
X(S)=\int_{\bR^d}S(x)X(dx) \quad  \mbox{for any $S \in \cP_{0,d}(\bR^d)$}.
\]

\medskip

{\em Notation:} 
The subscript ``$d$'' is used to emphasize that the integrands are deterministic. The set $\bR^d$ is included in the notation to avoid the confusion with the space $\cP_{0,d}$ introduced in \cite{BJ26}, which is the space of integrands that depend also on time. 

\medskip

Let $\cP_{+,d}(\bR^d)$ be the set of measurable functions $\varphi:\bR^d \to \bR$ such that
\begin{align*}
\| \varphi\|_{+}^2:= \int_{\bR^d}\int_{\bR^d}|\varphi(x)||\varphi(y)| f(x-y)dxdy<\infty.
\end{align*}

Since $\|\cdot\|_0 \leq \|\cdot\|_{+}$, it follows that
\[
\cP_{+,d}(\bR^d) \subset \cP_{0,d}(\bR^d).
\]

It can be proved that (see Lemma 3.1 of \cite{BJ26}), for any $\varphi \in \cP_{+,d}(\bR^d)$, 
\[
\varphi*\k \in L^2(\bR^d) \quad \mbox{and} \quad X(\varphi)=L(\varphi*\k).
\]


The space $\cP_{0,d}(\bR^d)$ contains distributions. Theorem 3.5.(1) of \cite{BGP12} shows that
\[
\overline{\cP}_d(\bR^d) \subset \cP_{0,d}(\bR^d),
\]
where 
\[
\overline{\cP}_d(\bR^d)
=\{S \in \cS'(\bR^d); \cF S \ \mbox{is a function}, \int_{\bR^d}|\cF S(\xi)|^2 \mu(d\xi)<\infty \}.
\]
See also Theorem 4.7 in \cite{B15}. The following remark gives more information about 
$\cP_{0,d}(\bR^d)$.

\begin{remark}
{\rm
Theorem 3.5.(2) of \cite{BGP12} shows that if $|h|^{-2}1_{\{|h|>0\}}$ is also tempered, then
\[
\overline{\cP}_d(\bR^d) = \cP_{0,d}(\bR^d).
\] 

If $k=B_{d,\alpha/2}$ for some $\alpha>0$, then $|h(\xi)|^{-2}1_{\{|h(\xi)|>0\}}=(1+|\xi|^2)^{\alpha/2}$ is tempered, and
\begin{align*}
\cP_{0,d}(\bR^d)&=\overline{\cP}_d(\bR^d)
=\{S \in \cS'(\bR^d); \cF S \ \mbox{is a function}, \int_{\bR^d}|\cF S(\xi)|^2 (1+|\xi|^2)^{-\alpha/2}d\xi<\infty \}\\
&=W^{-\frac{\alpha}{2},2}(\bR^d).
\end{align*}
where $W^{s,2}(\bR^d)=H^{s}(\bR^d)$ is the Sobolev space of order $s\in \bR$.

If $k=R_{d,\alpha/2}$ for some $\alpha\in (0,d)$, then $|h(\xi)|^{-2}1_{\{|h(\xi)|>0\}}=|\xi|^{\alpha}$ is tempered, and
\begin{align*}
\cP_{0,d}(\bR^d)&=\overline{\cP}_d(\bR^d)
=\{S \in \cS'(\bR^d); \cF S \ \mbox{is a function}, \int_{\bR^d}|\cF S(\xi)|^2 |\xi|^{-\alpha}d\xi<\infty \}\\
& \subseteq W^{-\frac{\alpha}{2},2}(\bR^d).
\end{align*}
}
\end{remark}

\medskip

This article is organized as follows. In Section \ref{section-mom}, we present a moment inequality for the stochastic integral with respect to $\widehat{N}$. In Section \ref{section-lin}, we study a linear SPDE with noise $X$, focusing on the existence of the solution and estimates for the moments. In Section \ref{section-mult}, we study an SPDE with multiplicative noise $X$ using Malliavin calculus on the Poisson space. The appendix contains some auxiliary results.

\section{Moment estimate}
\label{section-mom}

In this section, we present an inequality for the stochastic integral with respect to $\widehat{N}$ in the case of deterministic integrand. This inequality will be used in Section \ref{section-lin} below for estimating the moments of the solution of the linear equation with noise $X$.

\medskip

A process $(X_t)_{t\geq 0}$ is called {\em c\`adl\`ag} if it is right-continuous and has left limits.
We recall Rosenthal's inequality for c\`adl\`ag martingales (see e.g. Lemma 2.1 of \cite{DV90}). 

\begin{theorem}[Rosenthal's inequality]
\label{rosenthal}
Let $\{M_t\}_{t \geq 0}$ be a c\`adl\`ag square-integrable martingale with $M_0=0$ and $\langle M \rangle$ be its predictable quadratic variation. We denote by $(\Delta M)_t$ the jump size of $M$ at time $t$. Then for any $p \geq 2$, there exists a constant $\cB_p>0$ depending on $p$ such that for any $t>0$,
$$\|\sup_{s \leq t}|M_s|\|_p \leq \cB_p \Big(\| \langle M \rangle_t^{1/2}\|_p +\|\sup_{s \leq t}|(\Delta M)_s| \|_p\Big).$$
\end{theorem}

We consider the Hilbert space $\cH=L^2(Z,\cZ,\fm)$, where $({\bf Z},\cZ,\fm)$ is defined by \eqref{def-Z}.

We denote
\[
I_1(\phi)=
\int_{\bR^{d} \times \bR_0}\phi(x,z)\widehat{N}(dx,dz) \quad \mbox{for $\phi \in \cH$}.
\]

We denote by $\|\cdot\|_p$ the norm in $L^p(\Omega)$, for any $p\geq 1$. We have the following result.

\begin{theorem}
\label{ros-I1}
For any function $\phi \in \cH$ and for any $p\geq 2$,
\begin{equation}
\label{mom-I1}
\|I_1(\phi)\|_p \leq \cB_p (\|\phi\|_{\cH}+\|\phi\|_{L^p({\bf Z})}),
\end{equation}
where $\cB_p$ is the constant from Theorem \ref{rosenthal}.
\end{theorem}

\begin{proof}
{\em Step 1.} In this step, we proved that for any $B \in \cB_b(\bR^d)$ and $\Gamma \in \cB_b(\bR_0)$ fixed,
\begin{equation}
\label{mom-I2}
\|I_1(\phi 1_{B \times \Gamma})\|_p \leq \cB_p (\|\phi 1_{B \times \Gamma}\|_{\cH}+\|\phi 1_{B \times \Gamma}\|_{L^p({\bf Z})}).
\end{equation}
To avoid trivial situations, we assume that $|B|>0$ and $\nu(\Gamma)>0$.

For any $t>0$, let $A_t=\{x \in \bR^d;|x|\leq t\}$. We consider the filtration
\[
\cF_t=\sigma\big(\big\{N(A \times C); A \in \cB(\bR^d),A \subseteq A_t, C \in \cB_b(\bR_0)\big\}\big) \vee \cN,
\]
where $\cN$ is the class of $\bP$-negligible sets. We define $M_0=0$ and
\[
M_t:=I_1(\phi 1_{(B \cap A_t)\times \Gamma} )=\int_{B \times \Gamma} 1_{A_t}(x)\phi(x,z)\widehat{N}(dx,dz), \quad \ \mbox{for $t>0$}.
\]
Then $(M_t)_{t\geq 0}$ is a martingale with respect to $(\cF_t)_{t\geq 0}$ since for any $0\leq s<t$,
\[
M_t-M_s=\int_{B \times \Gamma} 1_{A_t\verb2\2 A_s}(x)\phi(x,z)\widehat{N}(dx,dz) \quad \mbox{is independent of $\cF_s$},
\]
and $\bE[M_t-M_s]=0$. Moreover, $(M_t)_{t\geq 0}$ is continuous in $L^2(\Omega)$: if $t_n \to t$, then
\[
\bE|M_{t_n}-M_t|^2 =\int_{B \times \Gamma} |1_{A_{t_n}} (x)-1_{A_t}(x)| |\phi(x,z)|^2 dx\nu(dz) \to 0.
\]
By Doob's regularity theorem (see e.g. Theorem 3.40 of \cite{PZ07}), any stochastically continuous submartingale has a c\`adl\`ag modification. Hence, $(M_t)_{t\geq 0}$ has a c\`adl\`ag modification, which we denote also $(M_t)_{t\geq 0}$. 

We apply Theorem \ref{rosenthal}. Clearly, the predictable quadratic variation of $(M_t)_{t\geq 0}$ is:
\[
\langle M\rangle_t=\int_{B \times \Gamma}1_{A_t}(x)|\phi(x,z)|^2 dx \nu(dz).
\]

Since the sets $B$ and $\Gamma$ are bounded, $M_t$ can be written as
\[
M_t=\int_{B \times \Gamma}1_{A_t}(x)\phi(x,z)N(dx,dz)-\int_{B \times \Gamma}1_{A_t}(x)\phi(x,z)dx\nu(dz)=:M_t^{(d)}-M_t^{(c)},
\]
where $(M_t^{(c)})_{t\geq 0}$ is continuous. To study the jumps of $(M_t^{(d)})_{t\geq 0}$, we recall the following equality in distribution:
\[
N|_{B \times \Gamma}\stackrel{d}{=}\sum_{i=1}^{\Lambda}\delta_{(X_i,Z_i)}=:N',
\]
 where $\Lambda$ is a Poisson random variable with mean $\lambda=|B|\nu(\Gamma)$, and $\{(X_i,Z_i)\}_{i\geq 1}$ are i.i.d. on $B \times \Gamma$ with law $\lambda^{-1}({\rm Leb} \times \nu)$, independent of $\Lambda$; see  e.g. Section 5.4.2 of \cite{resnick07}. Hence $I_1(\phi 1_{B \times \Gamma})$ has the same law (and hence, the same $p$-th moment) as $\int_{B \times \Gamma} \phi(x,z) \widehat{N}'(dx,dz)$. Therefore, without loss of generality, we can assume that
 \[
N|_{B \times \Gamma}=\sum_{i=1}^{\Lambda}\delta_{(X_i,Z_i)} \quad \mbox{a.s.}
\]
which means that with probability 1,
\[
M_t^{(d)}=\sum_{i=1}^{\Lambda}1_{A_t}(X_i) \phi(X_i,Z_i) = \sum_{i=1}^{\Lambda}1_{\{|X_i|\leq t\}} \phi(X_i,Z_i) \quad \mbox{for all $t>0$}.
\]

A deterministic function $f:\bR_{+} \to \bR$ given by $f(t)=\sum_{i=1}^k 1_{\{|x_i|\leq t\}} a_i$ for some $a_1,x_1,\ldots,a_n,x_n \in \bR$ with $|x_i| \not=|x_j|$ for all $i \not=j$, is c\`adl\`ag and has jumps at $t_i=|x_i|$ with jump sizes $\Delta f(t_i)=a_i$, for $i=1,\ldots,k$. In our case, 
\[
\bP(|X_i|=|X_j|)=\frac{1}{|B|^2}  \int_{B^2} 1_{\{|x|=|y|\}} dxdy=0 \quad \mbox{for any $i\not=j$}.
\]
Then the event $F=\bigcup_{i\not=j}\{|X_i|=|X_j|\}$ has probability 0. We will work on  $\widetilde{\Omega}=F^c$. On this event, the function $t\mapsto M_t^{(d)}$ has jumps at $T_i=|X_i|$ with jump sizes $\phi(X_i,Z_i)$, for any $i=1,\ldots,\Lambda$, if $\Lambda>0$, or has no jumps at all if $\Lambda=0$. The same thing is true about the function $t \mapsto M_t$. It follows that on $\widetilde{\Omega}$,
\begin{align*}
\sup_{s\leq t} |\Delta M_s|^p \leq  \sum_{s\leq t} |\Delta M_s |^p=\sum_{|X_i|\leq t} |\phi(X_i,Z_i)|^p=\int_{B \times \Gamma} 1_{\{|x|\leq t\}} |\phi(x,z)|^p N(dx,dz).
\end{align*}
Taking expectation, we infer that:
\[
\bE\big[ \sup_{s\leq t} |\Delta M_s|^p\big]\leq \int_{B \times \Gamma} 1_{A_t}(x) |\phi(x,z)|^p dx\nu(dz),
\]
and hence
\[
\big\|\sup_{s\leq t} |\Delta M_s\big\|_p \leq \left( \int_{B \times \Gamma} 1_{A_t}(x) |\phi(x,z)|^p dx\nu(dz) \right)^{1/p}.
\]

Using Theorem \ref{rosenthal}, we infer that for any $t>0$,
\[
\|M_t\|_p \leq \cB_p \left\{ \left( \int_{B \times \Gamma} 1_{A_t}(x) |\phi(x,z)|^2 dx\nu(dz) \right)^{1/2}+ \left( \int_{B \times \Gamma} 1_{A_t}(x) |\phi(x,z)|^p dx\nu(dz) \right)^{1/p} \right\}
\]

Next, we consider a sequence $(t_n)_{n\ge 1}$ such that $t_n \to \infty$. By the isometry property, $M_{t_n} \to I_1(\phi 1_{B \times \Gamma})$ in $L^2(\Omega)$ as $n \to \infty$. Hence,  $M_{t_n} \to I_1(\phi 1_{B \times \Gamma})$ a.s., along a subsequence. By Fatou's lemma,
\begin{align*}
\|I_1\big(\phi 1_{B \times \Gamma}\big)\|_p & \leq \liminf_{n \to \infty}\|M_{t_n}\|_p \\
& \leq \cB_p \left\{ \left( \int_{B \times \Gamma} |\phi(x,z)|^2 dx\nu(dz) \right)^{1/2}+ \left( \int_{B \times \Gamma} |\phi(x,z)|^p dx\nu(dz) \right)^{1/p} \right\}
\end{align*}
This proves \eqref{mom-I2}.

\medskip

{\em Step 2.} In this step, we prove \eqref{mom-I1}. For this, we consider a sequence $(B_k)_{k\geq 1}$ of sets in $\cB_b(\bR^d)$ such that $B_k \uparrow \bR^d$,
and a sequence $(\Gamma_k)_{k\geq 1}$ of sets in $\cB_b(\bR_0)$ such that $\Gamma_k \uparrow \bR_0$.

By {\em Step 1}, we know that for any $k\geq 1$,
\[
\|I_1(\phi 1_{B_k \times \Gamma_k})\|_p \leq \cB_p (\|\phi 1_{B_k \times \Gamma_k}\|_{\cH}+\|\phi 1_{B_k \times \Gamma_k}\|_{L^p({\bf Z})}).
\]
The conclusion follows letting $k \to \infty$, using Fatou's lemma on the left-hand side.

\end{proof}

\begin{corollary}
\label{cor-p-mom-L}
If $p\geq 2$ is such that $m_p<\infty$, then for any function $\varphi \in L^2(\bR^d)$,
\begin{equation}
\label{mom-L1}
\|L(\varphi)\|_p \leq \cB_p \Big( m_2^{1/2}\|\varphi\|_{L^2(\bR^d)}+ m_p^{1/p} \|\varphi\|_{L^p(\bR^d)}\Big) \leq \cC_p \Big( \|\varphi\|_{L^2(\bR^d)}+  \|\varphi\|_{L^p(\bR^d)}\Big),
\end{equation}
where $\cB_p$ is the constant from Theorem \ref{rosenthal} and $\cC_p=\cB_p (m_2^{1/2} \vee m_p^{1/p})$.
\end{corollary}

\begin{proof}
We apply Theorem \ref{ros-I1} to the function $\phi(x,z)=\varphi(x)z$.
\end{proof}


\begin{remark}
{\rm Applying \eqref{mom-L1} to $\varphi=1_A$ for $A \in \cB_b(\bR^d)$, we infer that for any $p\geq 2$ such that $m_p<\infty$,
\[
\|L(A)\|_p \leq \cB_p\Big\{\big(m_2|A|\big)^{1/2}+\big(m_p|A|\big)^{1/p}\Big\}.
\]
The same argument shows that if $\{Y(t)\}_{t\geq 0}$ is a classical L\'evy process with characteristic function
\[
\bE[e^{i\theta Y(t)}]=\exp\left\{t\int_{\bR_0}\big(e^{i \theta z}-1-i\theta z \big) \nu(dz)\right\}, \quad \theta \in \bR, 
\]
then for any $p\geq 2$ such that $m_p<\infty$,
\begin{equation}
\label{classic}
\|Y(t)\|_p \leq \cB_p\Big\{\big(m_2t\big)^{1/2}+\big(m_p t \big)^{1/p}\Big\}.
\end{equation}
Formula \eqref{classic} may be well-known, but we could not find a reference for it.
}
\end{remark}

\section{Linear equation}
\label{section-lin}

In this section we study the linear equation
\begin{equation}
\label{lin-eq}
\cL v(t,x)=\dot{X}(x) \quad t>0,\ x\in \bR^d,
\end{equation}
where $\cL$ is a second-order partial differential operator with coefficients that do not depend on the space variable. We denote by $G_t$ the fundamental solution of the equation $\cL u=0$.  

We recall that $\cO_C'(\bR^d)$ is the space of distributions with rapid decrease. The definition and key properties of such distributions are recalled in Appendix \ref{app-distr}. We will work under the following assumption.

\medskip

{\bf Hypothesis A.} For any $t>0$, $G_t \in \cO_C'(\bR^d)$.

\medskip

We will be interested in the following two examples:

\begin{example}[Heat equation] If $\cL=\frac{\partial}{\partial t}-\frac{1}{2}\Delta$ is the heat operator, 
then
\[
G_t(x)=\frac{1}{(2\pi t)^{d/2}} e^{-\frac{|x|^2}{2t}} \quad \mbox{and} \quad \cF G_t(\xi)=e^{-\frac{t|\xi|^2}{2}}.
\]
In this case, $G_t\in \cS(\bR^d)$ and Hypothesis A holds by Lemma \ref{SO-lem}.
\end{example}

\begin{example}[Wave equation] If $\cL=\frac{\partial^2}{\partial t^2}-\Delta$ is the wave operator, then
\[
G_t(x)= \frac{1}{2}1_{\{|x| <t\}} \quad \mbox{if \  $d=1$}, \quad 
G_t(x)=\frac{1}{2\pi}\frac{1}{\sqrt{t^2-|x|^2}}1_{\{|x|<t\}} \quad \mbox{if \ $d=2$},
\]
and
\[
G_t=\frac{1}{4\pi t}\sigma_t \quad \mbox{if \ $d=3$},
\]
where $\sigma_t$ is the surface measure on the sphere $\{x \in \bR^3; |x|=t\}$.

If $d \geq 4$ is even, $G_t$ is a distribution with compact support in $\bR^d$ given by: (see e.g. relation (38) on p. 80 of \cite{evans95})
$$G_t= \frac{1}{2 \cdot 4 \cdot \ldots \cdot n} \left(\frac{1}{t}\frac{\partial}{\partial t}
\right)^{\frac{d-2}{2}}(t^d\Upsilon_t), \quad (\Upsilon_t,\varphi)=\frac{1}{\alpha_{d}}\int_{B(0,1)}
\frac{\varphi(tz)}{\sqrt{t^2-|tz|^2}} dz,$$
where $\alpha_{d}=\frac{2\pi^{d/2}}{d\Gamma(d/2)}$ is the volume of the unit ball $B(0,1)$ in $\bR^d$. If
$d \geq 5$ is odd, $G_t$ is a distribution with compact support in $\bR^d$ given by: (see e.g. relation (31) on p. 77 of \cite{evans95})
$$G_t= \frac{1}{1 \cdot 3 \cdot 5 \ldots \cdot (d-2)}\left(\frac{1}{t}\frac{\partial}{\partial t}
 \right)^{\frac{d-3}{2}} (t^{d-2}\Sigma_t), \quad (\Sigma_t,\varphi)=\frac{1}{\omega_d}\int_{\partial B(0,1)}\varphi(tz)d\sigma(z),$$
where $\omega_d=d\alpha_d$ is the surface area of the unit sphere $\partial B(0,1)$ in $\bR^d$, and $\sigma$ is the surface measure on $\partial B(0,1)$.

It is known that
\[
\cF G_t(\xi)=\frac{\sin(t|\xi|)}{|\xi|} \quad \mbox{for any $d\geq 1$}.
\]

Hypothesis A holds since the space $\cE'(\bR^d)$ of distributions with compact support is included in $\cO_{C}'(\bR^d)$. (If $d\le 3$, we can view $G_t$ as a distribution. This distribution has compact support, since $G_t$ has compact support as a function, or as a measure.)

\end{example}

\subsection{Random field solution}

In this section, we show that under some conditions, equation \eqref{lin-eq} has a (unique) solution. Then, we show that these conditions are verified for the stochastic heat and wave equation.

\begin{definition}
{\rm
The process $\{v(t,x);t\geq 0,x\in \bR^d\}$ defined by
\begin{equation}
\label{defi-v}
v(t,x)=\int_0^t \int_{\bR^d}G_{t-s}(x-y) X(dy)ds
\end{equation}
is called a {\em random-field solution} of equation \eqref{lin-eq}, provided that the integral is well-defined. }
\end{definition}

To prove the existence of the solution, we use the following result. 
\begin{lemma}
\label{S-lem}
Let $S \in \cO_{C}'(\bR^d)$ be such that $\int_{\bR^d}|\cF S(\xi)|^2 \mu(d\xi)<\infty$. Then $S \in \overline{\cP}_{d}(\bR^d)$, $S * \k \in L^2(\bR^d)$ and 
\begin{equation}
\label{XL-int}
\int_{\bR^d}S(x)X(dx)=\int_{\bR^d}(S*\k)(x)L(dx) \quad \mbox{in $L^2(\Omega)$.}
\end{equation}
\end{lemma}

\begin{proof} The fact that $S \in \overline{\cP}_{d}$ is clear. Since $k$ induces a tempered distribution, $S*\k \in \cS'(\bR^d)$ (by Theorem \ref{GF1}) and $\cF(S*\k)=\cF S \cF \k$ (by Theorem \ref{GF2}). Note that $\cF(S*\k)\in L^2_{\bC}(\bR^d)$, since
\[
\int_{\bR^d}|\cF(S*\k)(\xi)|^2 d\xi=\int_{\bR^d}|\cF S(\xi)|^2 |\cF \k(\xi)|^2 d\xi=(2\pi)^d
\int_{\bR^d}|\cF S(\xi)|^2 \mu(d\xi)<\infty.
\]
By Lemma A.1 of \cite{BJ26}, it follows that $S*\k \in L^2(\bR^d)$.

It remains to prove \eqref{XL-int}. For this, we proceed by approximation. Since $S \in \cP_{0,d}(\bR^d)$, there exists a sequence $(\varphi_n)_{n\geq 1}$ in $\cD(\bR^d)$ such that $\|\varphi_n-S\|_0 \to 0$ as $n \to \infty$. Hence, 
\[
X(\varphi_n) \to X(S) \quad \mbox{in $L^2(\Omega)$}.
\]
On the other hand, $X(\varphi_n)=L(\varphi_n*\k) \to L(S *\k)$ in $L^2(\Omega)$, since
$\varphi_n *\k \to S *\k$ in $L^2(\bR^d)$: by Plancherel theorem,
\begin{align*}
\int_{\bR^d}|(\varphi_n* \k)(x) -(S* \k)(x)|^2dx&=\frac{1}{(2\pi)^d} \int_{\bR^d} |\cF \varphi_n(\xi)-\cF S(\xi)|^2 |\cF \k(\xi)|^2 d\xi \\
&= \|\varphi_n-S\|_0^2 \to 0 \quad \mbox{as $n\to \infty$}.
\end{align*}
Hence, $X(S)=L(S*\k)$ in $L^2(\Omega)$.
\end{proof}

We introduce the following hypotheses.

\medskip

\noindent {\bf Hypothesis B.} For any $t>0$, $\int_{\bR^d} |\cF G_t(\xi)|^2 \mu(d\xi)<\infty$. 

\medskip

\noindent {\bf Hypothesis C.}
For any $t>0$ and for any sequence $t_n \to t$, 
\[
\int_{\bR^d}|\cF G_{t_n}(\xi)-\cF G_t(\xi)|^2 \mu(d\xi) \to 0.
\]

\begin{remark}
\label{rem-HypB}
{\rm To verify Hypothesis C, it suffices to show that:\\
(i) the function $t \mapsto \cF G_t(\xi)$ is continuous on $\bR_{+}$, for any $\xi\in \bR^d$;\\
(ii) for any $t>0$, there exists $\e_t>0$ and a measurable function $k_t:\bR^d \to \bR_{+}$ with the property that $\int_{\bR^d}|k_t(\xi)|^2 \mu(d\xi)<\infty$, such that for any $h \in [-\e_t,\e_t]$ and $\xi \in \bR^d$,
\[
|\cF G_{t+h}(\xi)-\cF G_t(\xi)| \leq k_t(\xi).
\]
Then, apply the dominated convergence theorem.}
\end{remark}

\medskip

Under Hypothesis B, for $(t,x) \in \bR_{+} \times \bR^d$ fixed, we apply Lemma \ref{S-lem} to the distribution
$S(s)=G_{t-s}(x-\cdot)1_{[0,t]}(s) \in \cO_{C}'(\bR^d)$,
 since $\cF G_{t-s}(x-\cdot)(\xi)=e^{-i \xi \cdot x} \overline{\cF G_{t-s}(\xi)}$, and
\[
\int_{\bR^d} |\cF G_{t-s}(x-\cdot)(\xi)|^2 \mu(d\xi)<\infty.
\]
It follows that  for any $0\leq s\leq t$ and $x \in \bR^d$, $G_{t-s}(x-\cdot) \in \overline{\cP}_d(\bR^d)$, $G_{t-s}(x-\cdot)*\k \in L^2(\bR^d)$, and 
\begin{equation}
\label{def-Y}
Y^{(t,x)}(s):=\int_{\bR^d} G_{t-s}(x-y)X(dy)=\int_{\bR^d} \big(G_{t-s}(x-\cdot)*\k \big)(y)L(dy)=:Z^{(t,x)}(s) \quad \ \mbox{in $L^2(\Omega)$}.
\end{equation}
To study the joint measurability in $(\omega,s)$ of $Y^{(t,x)}$, we introduce another hypothesis:

\medskip

\noindent {\bf Hypothesis D.} The map $(t,\xi)\mapsto \cF G_t(\xi)$ is measurable on $\bR_{+} \times \bR^d$, and  
\[
\int_0^T \int_{\bR^d} |\cF G_t(\xi)|^2 \mu(d\xi)dt<\infty \quad \mbox{for any $T>0$}.
\]

\begin{lemma}
Under Hypotheses A to D, for any $(t,x)\in \bR_{+} \times \bR^d$,  $\{Y^{(t,x)}(s)\}_{s \in [0,t]}$ and $\{Z^{(t,x)}(s)\}_{s \in [0,t]}$ have jointly measurable modifications denoted by $\{\widetilde{Y}^{(t,x)}(s)\}_{s \in [0,t]}$, respectively $\{\widetilde{Z}^{(t,x)}(s)\}_{s \in [0,t]}$, and with probability 1,
\[
\int_0^t |\widetilde{Y}^{(t,x)}(s)|ds<\infty \quad \mbox{and} \quad \int_0^t |\widetilde{Z}^{(t,x)}(s)|ds<\infty.
\]
Moreover, for any $(t,x)\in \bR_{+} \times \bR^d$,
\[
v(t,x):=\int_0^t \widetilde{Y}^{(t,x)}(s)ds=\int_0^t \widetilde{Z}^{(t,x)}(s)ds \quad \mbox{in $L^2(\Omega)$}.
\]
The process $\{v(t,x); t\geq 0,x\in \bR^d\}$ is the random field solution of equation \eqref{lin-eq}.
\end{lemma}

\begin{proof}
By Hypothesis C, the processes $\{Y^{(t,x)}(s)\}_{s \in [0,t]}$ and $\{Z^{(t,x)}(s)\}_{s \in [0,t]}$  are continuous in $L^2(\Omega)$: if $s_n \to s \in [0,t]$, then
\begin{align*}
\bE|Y^{(t,x)}(s_n)-Y^{(t,x)}(s)|^2 &= 
\bE|Z^{(t,x)}(s_n)-Z^{(t,x)}(s)|^2 \\
&=m_2\int_{\bR^d}|\cF G_{t-s_n}(\xi)-\cF G_{t-s}(\xi)|^2 \mu(d\xi) \to 0 \quad \mbox{as $n \to \infty$}.
\end{align*}
 
By Proposition 3.21 of \cite{PZ07}, these processes have jointly measurable modifications denoted by $\{\widetilde{Y}^{(t,x)}(s)\}_{s \in [0,t]}$, respectively $\{\widetilde{Z}^{(t,x)}(s)\}_{s \in [0,t]}$. This means that for any $s \in [0,t]$,
\[
\widetilde{Y}^{(t,x)}(s)=Y^{(t,x)}(s)=Z^{(t,x)}(s)=\widetilde{Z}^{(t,x)}(s) \quad \mbox{a.s.}
\]

By Fubini theorem and Hypothesis D,
\[
\bE\int_0^t |\widetilde{Y}^{(t,x)}(s)|^2 ds=\int_0^t \bE|\widetilde{Y}^{(t,x)}(s)|^2 ds=m_2\int_0^t \int_{\bR^d}|\cF G_{t-s}(\xi)|^2 \mu(d\xi) ds <\infty.
\]
Hence, $\bE\int_0^t |\widetilde{Y}^{(t,x)}(s)|ds<\infty$, which implies that $\int_0^t |\widetilde{Y}^{(t,x)}(s)|ds<\infty$ a.s. Similarly, $\int_0^t |\widetilde{Z}^{(t,x)}(s)|ds<\infty$ a.s. 
Finally, by Cauchy-Schwarz inequality and Fubini's theorem,
\[
\bE
\left|\int_0^t \widetilde{Y}^{(t,x)}(s)ds- \int_0^t \widetilde{Z}^{(t,x)}(s)ds \right|^2 \leq t  \int_0^t 
\bE| \widetilde{Y}^{(t,x)}(s)ds-  \widetilde{Z}^{(t,x)}(s)|^2 ds=0
\]

\end{proof}

We will write
\begin{equation}
\label{def-v}
v(t,x)=\int_0^t \int_{\bR^d}G_{t-s}(x-y)X(dy)ds=\int_0^t \int_{\bR^d}\big( G_{t-s}(x-\cdot)*\k \big)(y) L(dy)ds,
\end{equation}
with the convention that the integrands $Y^{(t,x)}(s)$ and $Z^{(t,x)}(s)$ of the two $ds$ integrals are interpreted as $\widetilde{Y}^{(t,x)}(s)$, respectively $\widetilde{Z}^{(t,x)}(s)$.

\medskip

To verify that Hypotheses B, C and D hold for the fundamental solutions of the heat and wave equations,
we introduce {\em Dalang's condition}:
\[
(D) \quad \int_{\bR^d} \frac{1}{1+|\xi|^2}\mu(d\xi)<\infty.
\]

\begin{remark}
{\rm
a) If $k=R_{d,\alpha/2}$ for some $\alpha \in (0,d)$, then $f=R_{d,\alpha}$, $\mu(d\xi)=\frac{1}{(2\pi)^{d}} |\xi|^{-\alpha}d\xi$, and (D) holds if and only if
$\alpha \in (d-2,d)$.\\
b) If $k=B_{d,\alpha/2}$ for some $\alpha >0$, then $f=B_{d,\alpha}$, $\mu(d\xi)=\frac{1}{(2\pi)^{d}} (1+|\xi|^{2})^{-\alpha/2}d\xi$, and (D) holds if and only if
$\alpha >d-2$.\\
c) If $k=H_{d,\alpha/2}$ for some $\alpha >0$, then $f=H_{d,\alpha}$, $\mu(d\xi)=\frac{1}{(2\pi)^{d}} \exp(-\alpha|\xi|^2/2)d\xi$, and (D) holds for any $\alpha >0$.
}
\end{remark}

\begin{lemma}
\label{lem-add}
 Hypotheses B, C and D hold for the fundamental solutions of the heat and wave equations, for any $d\geq 1$, provided that (D) holds. In this case, the linear stochastic heat equation:
\begin{equation}
\label{heat-add}
\frac{\partial v}{\partial t}(t,x)=\frac{1}{2}\Delta v(t,x)+\dot{X}(x) \quad t>0,x\in \bR^d,
\end{equation}
with zero initial conditions has a random field solution, in any dimension $d\geq 1$. The same thing is true for the linear  stochastic wave equation:
\begin{equation}
\label{wave-add}
\frac{\partial^2 v}{\partial t^2}(t,x)=\Delta v(t,x)+\dot{X}(x) \quad t>0,x\in \bR^d,
\end{equation}
with zero initial conditions.
\end{lemma}

\begin{proof}
i) We consider first the heat equation. In this case, $G_t \in \cS(\bR^d)$ for all $t>0$.
Hypothesis B holds since 
\begin{align*}
\int_{|\xi|\leq 1}e^{-t|\xi|^2}\mu(d\xi) & \leq \int_{|\xi|\leq 1}\mu(d\xi)\leq 2\int_{|\xi|\leq 1} \frac{1}{1+|\xi|^2}\mu(d\xi), \\
\int_{|\xi|> 1}e^{-t|\xi|^2}\mu(d\xi) & \leq \int_{|\xi|> 1}\frac{1}{1+t|\xi|^2}\mu(d\xi) \leq \frac{1}{t\wedge 1}\int_{|\xi|> 1} \frac{1}{1+|\xi|^2}\mu(d\xi).
\end{align*}

To show that Hypothesis C holds, we use Remark \ref{rem-HypB}. 
Clearly, $t \mapsto e^{-t|\xi|^2/2}$ is continuous on $\bR_{+}$. To check (ii),
we fix $t>0$ and we let $h \in [-\e_t,\e_t]$, for some $\e_t>0$ (to be chosen later). If $h \in [0,\e_t]$, then
\begin{align*}
|\cF G_{t+h}(\xi)-\cF G_t(\xi)|=e^{-\frac{t|\xi|^2}{2}} \big(1-e^{-\frac{h|\xi|^2}{2}} \big)\leq e^{-\frac{t|\xi|^2}{2}},
\end{align*}
whereas if $h \in [-\e_t,0]$,
\begin{align*}
|\cF G_{t+h}(\xi)-\cF G_t(\xi)|=e^{-\frac{t|\xi|^2}{2}} \big(e^{-\frac{h|\xi|^2}{2}} -1\big)\leq e^{-\frac{t|\xi|^2}{2}}\big(e^{\frac{\e_t|\xi|^2}{2}} -1\big)\leq e^{-\frac{(t-\e_t)|\xi|^2}{2}}.
\end{align*}
Therefore, it suffices to choose $\e_t=\frac{t}{2}$ and $k_t(\xi)=e^{-\frac{t|\xi|^2}{4}}$. Finally, Hypothesis D holds, since there exist some constants $C_T^{(1)},C_{T}^{(2)}>0$ such that: (see e.g. Lemma 6.1 of \cite{sanz05})
\[
C_{T}^{(1)} \frac{1}{1+|\xi|^2} \leq \int_0^T e^{-t|\xi|^2}dt \leq C_T^{(2)}\frac{1}{1+|\xi|^2}.
\]

ii) We consider next the wave equation. In this case, $G_t \in \cE'(\bR^d)$ for all $t>0$.
Hypothesis B holds since
\begin{align*}
\int_{|\xi|\leq 1}\frac{\sin^2(t|\xi|)}{|\xi|^2}\mu(d\xi) & \leq t^2 \int_{|\xi|\leq 1}\mu(d\xi)\leq 2t^2 \int_{|\xi|\leq 1} \frac{1}{1+|\xi|^2}\mu(d\xi), \\
\int_{|\xi|> 1} \frac{\sin^2(t|\xi|)}{|\xi|^2} \mu(d\xi) & \leq \int_{|\xi|> 1}\frac{1}{|\xi|^2}\mu(d\xi) \leq 2\int_{|\xi|> 1} \frac{1}{1+|\xi|^2}\mu(d\xi).
\end{align*}

To show that Hypothesis C holds, we use Remark \ref{rem-HypB}.
Clearly, $t \mapsto \frac{\sin(t|\xi|)}{|\xi|}$ is continuous on $\bR_{+}$. To check (ii), we fix $t>0$ and we let $h \in [-\e_t,\e_t]$, for some $\e_t>0$ (to be chosen later). Then
\begin{align*}
|\cF G_{t+h}(\xi)-\cF G_t(\xi)|& = \frac{\big|\sin\big((t+h)|\xi|\big) -\sin\big(t|\xi|\big)\big|}{|\xi|}=\frac{2 \big|\sin\big(\frac{h|\xi|}{2} \big) \cos \big( \frac{(2t+h)|\xi|}{2}\big)\big|}{|\xi|} \\
& \leq \frac{2\big|\sin\big(\frac{h|\xi|}{2} \big)\big|}{|\xi|} 
\end{align*}
We bound the last expression by $|h| \leq \e_t$ if $|\xi|\leq 1$, and by $\frac{2}{|\xi|}$ if $|\xi|>1$. Therefore, it suffices to choose $\e_t=1$ and $k_t(\xi):=1_{\{|\xi| \leq 1\}}+\frac{2}{|\xi|}1_{\{|\xi| > 1\}}$.
Finally,  Hypothesis D holds since there exist some constants $C_T^{(3)},C_{T}^{(4)}>0$ such that: (see e.g. Lemma 6.1 of \cite{sanz05})
\[
C_{T}^{(3)} \frac{1}{1+|\xi|^2} \leq \int_0^T \frac{\sin^2(t|\xi|)}{|\xi|^2} dt \leq C_T^{(4)}\frac{1}{1+|\xi|^2}.
\]
\end{proof}

We will use the following result. 
\begin{theorem}[Stochastic Fubini Theorem]
\label{Fubini} Let $\Phi :[0,T] \times \bR^d \to \bR$ be such that:
\\
a) $\Phi(t,\cdot) \in L^2(\bR^d)$ for any $t \in [0,T]$, and the map $t \mapsto \Phi(t,\cdot)$ is continuous in $L^2(\bR^d)$;\\
b)  $\Phi  \in L^2([0,T] \times \bR^d)$.

Then the process $\{Z(t)=\int_{\bR^d}\Phi(t,x)L(dx)\}_{t \in [0,T] }$ has a jointly measurable modification denoted by $\{\widetilde{Z}(t)\}_{t \in [0,T]}$, 
\[
\int_0^T |\widetilde{Z}(t)|dt<\infty \ \mbox{a.s.},  \quad \int_0^T|\Phi(t,x)|dt<\infty \ \mbox{for almost all $x\in \bR^d$,}
\]
 the function $x\mapsto \int_0^T\Phi(t,x)dt$ belongs to $L^2(\bR^d)$, and
\begin{equation}
\label{Fubini-eq}
\int_0^T \widetilde{Z}(t) dt=\int_{\bR^d} \left( \int_0^T\Phi(t,x)dt \right) L(dx) \quad \mbox{in $L^2(\Omega)$}.
\end{equation}
\end{theorem}

\begin{proof} 
The process
$\{Z(t)\}_{t \in [0,T]}$ is continuous in $L^2(\Omega)$: if $t_n \to t \in [0,T]$, then
\[
\bE|Z(t_n)-Z(t)|^2=m_2 \| \Phi(t_n,\cdot)-\Phi(t,\cdot) \|_{L^2(\bR^d)}^2  \to 0 \quad \mbox{as $n \to \infty$}. 
\]
By Proposition 3.21 of \cite{PZ07}, $\{Z(t)\}_{t \in [0,T]}$ has a jointly measurable modification
$\{\widetilde{Z}(t)\}_{t \in [0,T]}$. By condition b),
\[
\bE \int_0^T |\widetilde{Z}(t)|^2 dt=\int_0^T \bE|Z(t)|^2 dt=m_2\int_0^T \int_{\bR^d}|\Phi(t,x)|^2 dxdt<\infty.
\]
Hence, $\bE \int_0^T |\widetilde{Z}(t)| dt<\infty$, which implies that $\int_0^T |\widetilde{Z}(t)| dt<\infty$ a.s.

\medskip

On the other hand, condition b) shows that the right-hand side of \eqref{Fubini-eq} is well-defined. The fact that $\int_{\bR^d}\int_0^T |\Phi(t,x)|^2 dxdt<\infty$ implies that $\int_0^T |\Phi(t,x)|^2 dt<\infty$ for almost all $x \in \bR^d$, and hence $\varphi(x):=\int_0^T\Phi(t,x)dt$ is well-defined for almost all $x \in \bR^d$. By H\"older's inequality,
$\int_{\bR^d}|\varphi(x)|^2 dx \leq T \int_0^T \int_{\bR^d}|\Phi(t,x)|^2 dxdt<\infty$, and hence, and $L(\varphi)$ is well-defined.

\medskip

Finally, it is enough to prove \eqref{Fubini-eq} for functions $\Phi$ of the form $\Phi(t,x)=1_A(t) 1_B(x)$ for $A \in \cB([0,T])$ and $B \in \cB_b(\bR^d)$. In this case, $Z(t)=1_A(t) L(B)$ is jointly measurable in $(\omega,t)$, so that we may take $\widetilde{Z}(t)=Z(t)$, and $\varphi(x)=|A|1_B(x)$. Hence,
\[
\int_0^T \widetilde{Z}(t)dt=\int_{\bR^d}\varphi(x) L(dx)=|A|L(B).
\]

\end{proof}

We introduce the following hypothesis.

\medskip

\noindent {\bf Hypothesis E.}
For any $(t,x)\in \bR_{+} \times \bR^d$, either one of the following conditions holds:\\
(i) the map $(s,y) \mapsto \big(G_{t-s}(x-\cdot)*\k\big)(y)$ is measurable on $[0,t] \times \bR^d$; \\
(ii) there exists a measurable function $g:[0,t] \times \bR^d \to \bR$ depending on $(t,x)$, such that 
\begin{equation}
\label{G*k-meas}
g(s,\cdot)= G_{t-s}(x-\cdot)*\k \ \mbox{in $L^2(\bR^d)$, \  for almost all $s \in [0,t]$}.
\end{equation}
We identify $G_{t-s}(x-\cdot)*\k$ with $g(s,\cdot)$, and we write $G_{t-s}(x-\cdot)*\k$ instead of $g(s,\cdot)$.

\begin{theorem}
Suppose that Hypotheses A to E hold.

Let $\{\widetilde{Z}^{(t,x)}(s)\}_{s \in [0,t]}$ be a jointly measurable modification of the process $\{Z^{(t,x)}(s)\}_{s \in [0,t]}$ defined in \eqref{def-Y}. Then, for any $(t,x)\in \bR_{+} \times \bR^d$, the function
$y \mapsto \int_0^t \big(G_{t-s}(x-\cdot)*\k\big)(y)ds$ belongs to  $L^2(\bR^d)$, and
\begin{equation}
\label{Fubini-v}
v(t,x)=\int_0^t \widetilde{Z}^{(t,x)}(s) ds=\int_{\bR^d} \left(\int_0^t \big(G_{t-s}(x-\cdot)*\k\big)(y)ds\right) L(dy).
\end{equation}
\end{theorem}

\begin{proof}
We fix $(t,x) \in \bR_{+} \times \bR^d$. We apply Theorem \ref{Fubini} to the function
\[
\Phi(s,y)=\big(G_{t-s}(x-\cdot)*\k\big)(y) \quad \mbox{for $s \in [0,t],y \in \bR^d$}.
\]

Hypotheses B and C ensure that condition a) is satisfied: $\Phi(s,\cdot) \in L^2(\bR^d)$ for all $s \in [0,t]$ (by Lemma \ref{S-lem}), and 
if $s_n \to s \in [0,t]$, then by Plancherel's theorem,
\[
\int_{\bR^d} |\Phi(s_n,y)-\Phi(s,y)|^2 dy =\frac{1}{(2\pi)^d} \int_{\bR^d} |\cF G_{t-s_n}(\xi)-\cF G_{t-s}(\xi)|^2 |\cF \k (\xi)|^2 d\xi \to 0.
\]

Hypotheses D and E ensure that condition b) is satisfied: $(s,y) \mapsto \Phi(s,y)$ is measurable on $[0,t] \times \bR^d$, and by Plancherel's theorem,
\[
\int_0^t \int_{\bR^d}|\Phi(s,y)|^2 dyds=\frac{1}{(2\pi)^d}\int_0^t \int_{\bR^d} |\cF G_{t-s}(\xi)|^2 |\cF \k(\xi)|^2 d\xi<\infty.
\]
\end{proof}

\begin{remark}
{\rm The integral on the {\em left} side of \eqref{Fubini-v} is well-defined, if Hypotheses B to D hold. The integral on the {\em right} side of \eqref{Fubini-v} is well-defined under
Hypotheses D and E.
}
\end{remark}

\begin{remark}
{\rm
From relations \eqref{def-v} and \eqref{Fubini-v}, we deduce that the random field solution $v$ of equation \eqref{lin-eq} has the representation:
\[
v(t,x)=\int_{\bR^d \times \bR_0} z \left( \int_0^t \big(G_{t-s}(x-\cdot)*\k\big)(y) ds\right) \widehat{N}(dy,dz).
\]
In Section \ref{section-mult} below, we will see that the integral above coincides with the first term in the Poisson-chaos expansion of the solution of the equation $\cL u(t,x) =u(t,x)\dot{X}(x)$.
}
\end{remark}

We check the validity of Hypothesis E for the heat and wave equations.

\begin{lemma}
a) Hypothesis E.(i) holds for the fundamental solution of the heat equation with $d\geq 1$, and for the fundamental solution of the wave equation with $d\leq 2$.\\
b) Hypothesis E.(ii) holds for the fundamental solution of the wave equation with $d\geq 3$, provided that (D) holds.
\end{lemma}

\begin{proof} 
Without loss of generality, we take $t=T$, $x=0$, and we prove the statement for the map $(t,x) \mapsto(G_t * \k)(x)$ instead of $(t,x) \mapsto (G_{T-t} * \k)(x)$.

a) The map $(t,x) \mapsto (G_t*\k)(x)$ is measurable on $\bR_{+} \times \bR^d$ by Fubini theorem, since $(t,x,y) \mapsto G_t(x-y)\k(y)$ is measurable, and
\[
\big(G_t* \k \big)(x)=\int_{\bR^d}G_t(x-y) \k(y)dy.
\]

b) We apply Proposition 1.2.25 of \cite{HNVW} to the spaces $S=[0,T],T=\bR^d,X=\bR$ and the function $F:[0,T] \to L^2(\bR^d)$ given by $F(t)=G_t*\k$. By Lemma \ref{lem-add}, Hypotheses B and C are satisfied, since (D) holds.
Then
$F$ is well-defined, since by Hypothesis B,
\[
\|G_t*\k\|_{L^2(\bR^d)}^2=\int_{\bR^d}|\cF G_t(\xi)|^2 \mu(d\xi)<\infty \quad \mbox{for any $t \in [0,T]$}.
\]
Moreover, by Hypothesis C, $F$ is continuous, since if $t_n \to t\in [0,T]$, then
\[
\|F(t_n)-F(t)\|_{L^2(\bR^d)}^2=\int_{\bR^d} |\cF G_{t_n}(\xi)-\cF G_t(\xi)|^2 \mu(d\xi) \to 0.
\]
Hence, $F$ is measurable with respect to the Borel $\sigma$-field of $L^2(\bR^d)$.
Since $L^2(\bR^d)$ is separable, $F$ is strongly measurable (by Corollary 1.1.10 of \cite{HNVW}), and because the Lebesgue measure $\lambda_d$ on $\bR^d$ is $\sigma$-finite, $F$ is strongly $\lambda_d$-measurable (by Proposition 1.1.16, {\em ibid}). Moreover,
\[
\int_0^T \|F(t)\|_{L^2(\bR^d)}^2 dt =\int_0^T \int_{\bR^d}|(G_t*\k)(x)|^2 dxdt=\int_0^T \int_{\bR^d}|\cF G_t(\xi)|^2 \mu(d\xi)dt<\infty,
\]
and hence, $\int_0^T \|F(t)\|_{L^2(\bR^d)} dt <\infty$. By Proposition 1.2.2, {\em ibid.}, $F$ is Bochner integrable. By Proposition 1.2.25, {\em ibid}, there exists a measurable function $g:[0,T] \times \bR^d \to \bR$ such that $g(t,\cdot)= G_{t}*\k$ in $L^2(\bR^d)$, for almost all $t \in [0,T]$.
\end{proof}

\begin{corollary}
\label{cor-Fub}
Let $\cL$ be the heat or wave operator in dimension $d\geq 1$. If (D) holds, then the random field solution $v$ of equation \eqref{lin-eq} has representation \eqref{Fubini-v}.
\end{corollary}

\subsection{Moments of the solution}

In this section, we study the existence of the moments of the random field solution of the linear stochastic heat and wave equations, in the case of some examples of kernels $\k$.

\medskip

As in \cite{BJ26}, for any $t>0$ and $p\geq 2$, we define
\[
J_p(t)=\|G_t*\k\|_{L^p(\bR^d)}^2.
\]

We have the following general result.

\begin{lemma}
\label{mom-v-lem}
Suppose that Hypotheses B and C hold. Let $v$ be the random field solution of equation \eqref{lin-eq} and $p\geq 2$ be such that $m_p<\infty$. Then for any $t>0$, $x \in \bR^d$, we have:
\[
\|v(t,x)\|_p \leq \cC_p \int_0^t \Big(\big(J_2(s) \big)^{1/2}+ \big(J_p(s) \big)^{1/2}\Big)ds,
\]
where $\cC_p$ is the constant from Corollary \ref{cor-p-mom-L}.
\end{lemma}

\begin{proof}
Using representation \eqref{def-v} and Minkowski's inequality, we have: 
\[
\|v(t,x)\|_p \leq \int_0^t \left\| \int_{\bR^d} \big(G_{t-s}(x-\cdot)*\k \big)(y) L(dy)\right\|_p ds.
\]
Applying Corollary \ref{cor-p-mom-L}, we infer that:
\begin{align*}
\left\| \int_{\bR^d} \big(G_{t-s}(x-\cdot)*\k \big)(y) L(dy)\right\|_p & \leq \cC_p  \Big( \|G_{t-s}(x-\cdot)*\k\|_{L^2(\bR^d)}+ \|G_{t-s}(x-\cdot)*\k\|_{L^p(\bR^d)} \Big)\\
&=\cC_p \Big( \big(J_2(s) \big)^{1/2}+ \big(J_p(s) \big)^{1/2} \Big).
\end{align*}
\end{proof}

By focusing on the heat and wave equations and particular kernels $\k$, we will obtain explicit conditions on $p$ which guarantee that the solution $v(t,x)$ has a finite $p$-th moment. 

\medskip

We will need the following result. Recall that the kernels $H_{d,\alpha}$ and $R_{d,\alpha}$ are given by Examples \ref{heat-ex} and \ref{Riesz-ex}.

\begin{theorem}[Theorem 4.4 of \cite{BJ26}]
\label{th44}
a) If $\k =H_{d,\alpha/2}$ for some $\alpha>0$, then
\[
J_p(t) \leq C \int_{\bR^d}|\cF G_t(\xi)|^2 e^{-\frac{\alpha|\xi|^2}{2}}d\xi \quad \mbox{for any $p>0$}.
\]
b) If $\k=R_{d,\alpha/2}$ for some $\alpha \in (0,d)$, then
\[
J_p(t) \leq C \|G_t\|_{L^q(\bR^d)}^2 \quad \mbox{for any} \quad p>\frac{2d}{2d-\alpha},
\]
provided that $G_t \in L^q(\bR^d)$, where $\frac{1}{q}=\frac{1}{p}+\frac{\alpha}{2d}$.\\
c) If $\k \in L^1(\bR^d)$ (in particular if $\k=B_{d,\alpha/2}$ for some $\alpha>0$), then
\[
J_p(t) \leq C \|G_t\|_{L^p(\bR^d)}^2 \quad \mbox{for any}  \quad p\geq 1,
\]
provided that $G_t \in L^p(\bR^d)$. (Here $C>0$ is a constant which depends on $(d,\alpha,p)$.)
\end{theorem}

We now examine the existence of the moments for the solution of the linear stochastic heat and wave equations. We note that the conditions that we obtain below are {\em weaker} than those encountered in Examples 4.7-4.12 for the linear heat and wave equations driven by a {\em time-dependent} L\'evy colored noise with the same spatial covariance kernel $\k$.

\begin{example}[Heat Equation]
\label{heat-ex}
In this case, $G_t^q(x)=c_q t^{\frac{d}{2}(1-q)}G_{t/q}(x)$, and hence
\begin{equation}
\label{heat-G-Lq}
\|G_t\|_{L^q(\bR^d)}=c_q^{\frac{1}{q} }t^{\frac{d}{2}\big(\frac{1}{q}-1\big)} \quad \mbox{for any $q>0$}.
\end{equation}
Assume that (D) holds, and let $v$ be the random field solution of \eqref{heat-add}. 

\begin{description}

\item[(i)] Suppose that \fbox{$\k=H_{d,\alpha/2}$} for some $\alpha>0$. By Theorem \ref{th44}.a),
\[
J_p(t) \leq C \int_{\bR^d} e^{-t|\xi|^2}e^{-\frac{\alpha|\xi|^2}{2}}d\xi=C\Big(t+\frac{\alpha}{2}\Big)^{-d/2} \leq C \quad \mbox{for any $p>0$}.
\]
By Lemma \ref{mom-v-lem},
\[
\|v(t,x)\|_p \leq Ct \quad \mbox{for any $p\geq 2$ such that $m_p<\infty$}.
\]

\item[(ii)] Suppose that \fbox{$\k=R_{d,\alpha/2}$} for some $\alpha \in \big((d-2)\vee 0,d\big)$.  By Theorem \ref{th44}.b) and \eqref{heat-G-Lq},
\begin{equation}
\label{heat-J-R}
J_p(t) \leq C t^{d\big(\frac{1}{p}+\frac{\alpha}{2d}-1 \big)} \quad \mbox{for any $p\geq 2\big(>\frac{2d}{2d-\alpha}\big)$}.
\end{equation}
By Lemma \ref{mom-v-lem},
\begin{align}
\nonumber
\|v(t,x)\|_p & \leq C \int_0^t \Big(s^{\frac{d}{2}\big(\frac{1}{2}+\frac{\alpha}{2d}-1\big)}+s^{\frac{d}{2}\big(\frac{1}{p}+\frac{\alpha}{2d}-1\big)} \Big) ds \\
\label{heat-v-R}
&= C t^{\frac{d}{2}\big(\frac{\alpha}{2d}-1\big)+1} \big(t^{\frac{d}{4}}+t^{\frac{d}{2p}} \big),
\end{align}
for any $p\geq 2$ such that $m_p<\infty$ and $\frac{d}{2}\big(\frac{1}{p}+\frac{\alpha}{2d}-1 \big)+1>0$, i.e.
\begin{equation}
\label{p-heat-R}
p(2d-\alpha-4)<2d.
\end{equation}

If $\alpha \geq 2d-4$, condition \eqref{p-heat-R} holds for any $p>0$. Since we must also have $\alpha<d$, this case can only occur if $2d-4<d$, i.e. $d<4$.

If $\alpha<2d-4$, condition \eqref{p-heat-R} holds for any $p<\frac{2d}{2d-\alpha-4}$. (The fact that $p\geq 2$ does not impose any additional restrictions, since $2<\frac{2d}{2d-\alpha-4}$ is equivalent to $\alpha>d-4$.) Since we must have $\alpha> d-2$, this case can only occur if $d-2<2d-4$, i.e. $d>2$.

In summary, we obtain that \eqref{heat-v-R} holds in the following cases:
\begin{align*}
d\leq 2, & \quad \alpha \in (0,d), \quad p\geq 2 \ \mbox{such that} \ m_p<\infty,\\
d=3, & \quad \alpha \in [2,3), \quad p\geq 2 \ \mbox{such that} \ m_p<\infty, \\
d=3, & \quad \alpha \in (1,2), \quad 2\leq p<\frac{6}{2-\alpha} \ \mbox{such that} \ m_p<\infty, \\
d\geq 4, & \quad \alpha \in (d-2,d), \quad 2\leq p<\frac{2d}{2d-\alpha-4} \ \mbox{such that} \ m_p<\infty .
\end{align*}

\item[(iii)] Suppose that \fbox{$\k\in L^1(\bR^d)$}. By Theorem \ref{th44}.c) and \eqref{heat-G-Lq},
\begin{equation}
\label{heat-J-L}
J_p(t) \leq C t^{d\big(\frac{1}{p}-1 \big)} \quad \mbox{for any $p\geq 1$}.
\end{equation}
By Lemma \ref{mom-v-lem},
\begin{align}
\nonumber
\|v(t,x)\|_p & \leq C \int_0^t \Big(s^{\frac{d}{2}\big(\frac{1}{2}-1 \big)} + s^{\frac{d}{2}\big(\frac{1}{p}-1 \big)}  \Big) ds\\
\label{heat-v-L1}
&=C t^{-\frac{d}{2}+1} \big(t^{\frac{d}{4}}+t^{\frac{d}{2p}} \big),
\end{align}
for any $p\geq 2$ such that $m_p<\infty$ and $\frac{d}{2}\big(\frac{1}{p}-1 \big)+1>0$, i.e.
\begin{equation}
\label{p-heat-L1}
p(d-2)<d.
\end{equation}
If $d\leq 2$, condition \eqref{p-heat-L1} holds for any $p>0$, whereas if $d>2$, condition \eqref{p-heat-L1} reduces to $p<\frac{d}{d-2}$. Since we must also have $p\geq 2$, this case can occur only if $2<\frac{d}{d-2}$, i.e. $d<4$.
In summary, \eqref{heat-v-L1} holds in the following cases:
\begin{align*}
d\leq 2, & \quad p\geq 2 \ \mbox{such that} \ m_p<\infty,\\
d=3, & \quad 2 \leq p<3 \ \mbox{such that}  \ m_p<\infty.
\end{align*}

\end{description}

\end{example}

\begin{example}[Wave Equation]
\label{wave-ex}
In this case, $G_t$ is a function only in dimension $d\leq 2$, and
\begin{equation}
\label{wave-G-Lq}
\|G_t\|_{L^q(\bR^d)}=
\left\{
\begin{array}{ll}
C t^{1/q} & \mbox{if $d=1$ and $q>0$}, \\
C t^{\frac{2}{q}-1}& \mbox{if $d=2$ and $q\in (0,2)$}, \\
\infty & \mbox{if $d=2$ and $q \geq 2$}.
\end{array} \right.
\end{equation}

Assume that (D) holds, and let $v$ be the solution of \eqref{wave-add}. 

\begin{description}

\item[(i)] Suppose that \fbox{$\k=H_{d,\alpha/2}$} for some $\alpha>0$. By Theorem \ref{th44}.a),
\[
J_p(t) \leq C \int_{\bR^d} \frac{\sin^2(t|\xi|)}{|\xi|^2}e^{-\frac{\alpha|\xi|^2}{2}}d\xi \leq C t^2 \quad \mbox{for any $p>0$}.
\]
By Lemma \ref{mom-v-lem},
\[
\|v(t,x)\|_p \leq Ct^2 \quad \mbox{for any $p\geq 2$ such that $m_p<\infty$}.
\]

\item[(ii)] Suppose that \fbox{$d\leq 2$} and \fbox{$\k=R_{d,\alpha/2}$} for some $\alpha \in (0,d)$. By Theorem \ref{th44}.b) and \eqref{wave-G-Lq},
\[
J_p(t) \leq 
\left\{
\begin{array}{ll}
C t^{\frac{2}{p}+\alpha} & \mbox{for any $p\geq 2$, \quad \quad \quad \quad if $d=1$}, \\
C t^{\frac{4}{p}+\alpha-2} & \mbox{for any $2 \leq p<\frac{4}{2-\alpha}$, \quad if $d=2$}.
\end{array} \right.
\]
By Lemma \ref{mom-v-lem},
\[
\|v(t,x)\|_p \leq 
\left\{
\begin{array}{ll}
C t^{\frac{\alpha}{2}+1} \big(t^{\frac{1}{2}}+t^{\frac{1}{p}} \big) &  \mbox{for any $p\geq 2$ with $m_p<\infty$, \quad \quad \quad \ \ if $d=1$}, \\
C t^{\frac{\alpha}{2}} \big(t+t^{\frac{2}{p}} \big)  & \mbox{for any $2 \leq p<\frac{4}{2-\alpha}$ with $m_p<\infty$, \quad if $d=2$}.
\end{array} \right.
\]

\item[(iii)] Suppose that \fbox{$d=1$} and \fbox{$\k\in L^1(\bR^d)$}. By Theorem \ref{th44}.c) and \eqref{wave-G-Lq},
\[
J_p(t) \leq C t^{\frac{2}{p}} \quad \mbox{for any $p>0$}.
\]
By Lemma \ref{mom-v-lem},
\[
\|v(t,x)\|_p \leq C t \big(t^{\frac{1}{2}}+t^{\frac{1}{p}} \big) \quad \mbox{for all $p\geq 2$ such that $m_p<\infty$}.
\]

\end{description}
\end{example}

\begin{remark}
{\rm There is another method for estimating the moments of $v(t,x)$. Using H\"older's inequality and Corollary \ref{cor-p-mom-L}, we obtain that for any $p\geq 2$ such that $m_p<\infty$,
\begin{align}
\nonumber
\bE|v(t,x)|^p & \leq t^{p-1} \int_0^t \bE \left|\int_{\bR^d} \big(G_{t-s}(x-\cdot)*\k\big)(y)L(dy) \right|^p ds\\
\label{method2}
& \leq t^{p-1} \cC_p^p \int_0^t \Big( \big(J_2(s) \big)^{p/2}+ \big(J_p(s) \big)^{p/2}\Big)ds.
\end{align}
This method gives stronger conditions. To see this, consider the stochastic heat equation. 

If $\k=R_{d,\alpha/2}$ for some $\alpha \in \big((d-2)\vee 0, d\big)$, then using estimate \eqref{heat-J-R} for $J_p(t)$ and \eqref{method2}, 
\[
\bE|v(t,x)|^p \leq t^{p-1} \cC_p^p \int_0^t \Big(s^{\frac{dp}{2}\big(\frac{1}{2}+\frac{\alpha}{2d}-1\big)}+s^{\frac{dp}{2}\big(\frac{1}{p}+\frac{\alpha}{2d}-1\big)} \Big) ds.
\]
The last integral is finite provided that 
\begin{equation}
\label{cond-p1}
\frac{dp}{2}\Big(\frac{1}{p}+\frac{\alpha}{2d}-1 \Big)+1>0,
\end{equation}
which is a {\em stronger} requirement than condition $\frac{d}{2}\big(\frac{1}{p}+\frac{\alpha}{2d}-1 \big)+1>0$ encountered in Example \ref{heat-ex}.(ii), since $\frac{1}{p}<1-\frac{\alpha}{2d}$ (due to the fact that $p\geq 2>\frac{2d}{2d-\alpha}$). We note also that \eqref{cond-p1} coincides with the condition encountered in Example 4.8 of \cite{BJ26} for the time-dependent L\'evy colored noise.

Similarly, if $\k \in L^1(\bR^d)$, then using estimate \eqref{heat-J-L} for $J_p(t)$ and \eqref{method2}, we see that
\[
\bE|v(t,x)|^p \leq t^{p-1} \cC_p^p \int_0^t  \Big(s^{\frac{dp}{2}(\frac{1}{2}-1)}+s^{\frac{dp}{2}(\frac{1}{p}-1)} \Big) ds.
\]
The last integral is finite provided that $\frac{dp}{2}\big(\frac{1}{p}-1 \big)+1>0$ (i.e. $p<1+\frac{2}{d}$), a condition which is {\em stronger} than condition $\frac{d}{2}\big(\frac{1}{p}-1 \big)+1>0$ encountered in Example \ref{heat-ex}.(iii), and which coincides with the one encountered in Example 4.9 of \cite{BJ26} for the time-dependent L\'evy colored noise.
}
\end{remark}

\subsection{Distribution-valued solution}

In this section, we introduce another concept of solution for the linear equation, and we analyze its relationship with the random field solution, similarly to Section 4 of \cite{dalang99}.

\medskip

The new definition is motivated by the following formal calculation. We fix $t>0$ and $\phi \in \cD(\bR^d)$. We multiply \eqref{defi-v} by $\phi(x)$ and we integrate $dx$. Using a formal stochastic Fubini theorem, we obtain:
\begin{align*}
\int_{\bR^d}v(t,x)\phi(x)dx=\int_0^t \int_{\bR^d} \left(\int_{\bR^d} G_{t-s}(x-y)\phi(x)dx\right) X(dy)ds
=\int_0^t X\big(\phi* \widetilde{G}_{t-s}\big)ds.
\end{align*}

Under Hypothesis A, $X\big(\phi* \widetilde{G}_{t-s}\big)$ is well-defined in $L^2(\Omega)$ for any $s \in [0,t]$, since $\phi* \widetilde{G}_{t-s} \in \cS(\bR^d)$ and a function in $\cS(\bR^d)$ induces a distribution which belongs to $\overline{\cP}_d(\bR^d)$. But the map $s \mapsto X\big(\phi* \widetilde{G}_{t-s}\big)$ may not be measurable or integrable on $[0,t]$. To avoid this problem, we will work with a modification, which we define below.

\medskip

\begin{lemma}
If Hypothesis A holds, then $\int_{\bR^d}|\cF(\phi* \widetilde{G}_{t})(\xi)|^2 \mu(d\xi)
<\infty$ for any $t>0$ and $\phi \in \cD(\bR^d)$.
\end{lemma}

\begin{proof} By Theorem A.2 and Hypothesis A, $\cF G_t \in \mathcal O_M$. Hence, there exist constants $C,k>0$ (depending on $t$) such that $|\cF G_t(\xi)| \le C(1+|\xi|^2)^k$ for all $\xi\in \bR^d$.

On the other hand, by Assumption A1.(b), $|\cF \k|^2$ is tempered, so there exists $\ell>0$ such that
\[
\int_{\bR^d} \left(\frac{1}{1+|\xi|^2}\right)^{\ell}\,|\cF \k(\xi)|^2\,d\xi < \infty.
\]
Therefore,
\begin{equation*}
\begin{aligned}
& \int_{\bR^d} |\cF(\phi*\widetilde{G}_{t})(\xi)|^2\,\mu(d\xi)
 = \int_{\bR^d} |\cF\phi(\xi)|^2\,|\cF G_{t}(\xi)|^2\,\mu(d\xi) \\
& \quad \quad \quad \le C \int_{\bR^d} |\cF\phi(\xi)|^2 (1+|\xi|^2)^{2k}\,|\cF \k(\xi)|^2\,d\xi \\
& \quad \quad \quad \le C \sup_{\xi\in\bR^d}\Big(|\cF\phi(\xi)|^2 (1+|\xi|^2)^{2k+\ell}\Big)
\int_{\bR^d} \left(\frac{1}{1+|\xi|^2}\right)^{\ell}\,|\cF \k(\xi)|^2\,d\xi < \infty,
\end{aligned}
\end{equation*}
where for the last line we used the fact that $\cF\phi\in \cS(\bR^d)$.
\end{proof}

We fix $s\in [0,t]$. By Lemma \ref{SO-lem}, the smooth function $\phi* \widetilde{G}_{t-s}$ induces a distribution in $\cO_{C}'(\bR^d)$. 
Applying Lemma \ref{S-lem} to $S=\phi* \widetilde{G}_{t-s}$, we infer that:
\[
Y^{(t,\phi)}(s):=\int_{\bR^d} \big( \phi * \widetilde{G}_{t-s}\big)(y) X(dy)=\int_{\bR^d} \big( \phi * \widetilde{G}_{t-s}* \k \big)(y) L(dy)=: Z^{(t,\phi)}(s) \quad \mbox{in $L^2(\Omega)$}.
\]

We introduce the following hypothesis.

\medskip

\noindent {\bf Hypothesis F.} (i) The function $t \mapsto \cF G_t(\xi)$ is continuous on $\bR_{+}$, for any $\xi \in \bR^d$; \\
(ii) $|\cF G_t(\xi)|\leq C_t$ for any $\xi \in \bR^d$, where $C_t>0$ is non-decreasing in $t$.

\medskip

Due to part (i) of Hypothesis F, processes $\{Y^{(t,\phi)}(s)\}_{s \in [0,t]}$ and $\{Z^{(t,\phi)}(s)\}_{s \in [0,t]}$ are continuous in $L^2(\Omega)$, since if $s_n \to s \in [0,t]$, then
\begin{align*}
\bE|Y^{(t,\phi)}(s_n)-Y^{(t,\phi)}(s)|^2&=\bE|Z^{(t,\phi)}(s_n)-Z^{(t,\phi)}(s)|^2=m_2\big \|\phi* \big(\widetilde{G}_{t-s_n}- \widetilde{G}_{t-s}\big)*\k\big\|_{L^2(\bR^d)}^2\\
&=m_2 \int_{\bR^d}|\cF \phi(\xi)|^2 |\cF G_{t-s_n}(\xi)-\cF G_{t-s}(\xi)|^2 \mu(d\xi) \to 0.
\end{align*}
Hence, there exists a jointly measurable process $\{\widetilde{Y}^{(t,\phi)}(s)\}_{s \in [0,t]}$ such that $Y^{(t,\phi)}(s)=\widetilde{Y}^{(t,\phi)}(s)$ a.s. for any $s\in [0,t]$. Moreover, part (ii) of Hypothesis F allows us to infer that
\[
\bE\int_0^t |\widetilde{Y}^{(t,\phi)}(s)|^2 ds=m_2 \int_0^t \int_{\bR^d}|\cF \phi(\xi)|^2 |\cF G_{t-s}(\xi)|^2 \mu(d\xi)ds<\infty.
\]
Hence, with probability 1, $\int_0^t |\widetilde{Y}^{(t,\phi)}(s)|^2 ds<\infty$ and $\int_0^t |\widetilde{Y}^{(t,\phi)}(s)| ds<\infty$. 

In summary, if Hypotheses A and F hold, we define for any $t>0$ and $\phi \in \cD(\bR^d)$,
\begin{equation}
\label{def-vtp}
v^{(t)}(\phi):=\int_0^t \widetilde{Y}^{(t,\phi)}(s)ds.
\end{equation}
We say that  $\{v^{(t)};t\geq 0 $\}  is a {\em distribution solution} of \eqref{lin-eq}.
The following result gives the relation between the distribution solution and the random field solution.

\begin{theorem}
Suppose that Hypotheses A and F hold.
Let
\[
H_t(\xi):=\int_0^t \cF G_{t-s}(\xi)ds 
\qquad t>0,\ \xi\in \bR^d.
\]
In order that there exists a jointly measurable, locally mean-square bounded process
$V=\{V(t,x);t\geq 0,\ x\in \bR^d\}$, continuous in mean square, such that for any $t>0$ and $\phi \in \cD(\bR^d)$,
\begin{equation}
\label{v-V}
v^{(t)}(\phi)=\int_{\bR^d}V(t,x)\phi(x)dx
\quad  \mbox{a.s.},
\end{equation}
it is necessary and sufficient that
\begin{equation}
\label{cond1}
\sup_{t \in [0,T]}\int_{\bR^d}|H_t(\xi)|^2\mu(d\xi)<\infty,
\end{equation}
and
\begin{equation}
\label{cond2}
\lim_{t\to s}\int_{\bR^d}|H_t(\xi)-H_s(\xi)|^2\mu(d\xi)= 0.
\end{equation}
If in addition, the random-field solution $v$ of \eqref{lin-eq} exists and has representation \eqref{Fubini-v},
then $V$ is a modification of $v$.
\end{theorem}

\begin{proof}
(Necessity). Assume that there exists a process $V$ satisfying the conditions of the theorem. We proceed as in the proof of Theorem 11 of \cite{dalang99}. Fix $x_0\in \bR^d$. Let $\psi_n(x)=n^d\psi(nx),$ where $\psi\in \cD(\bR^d)$, $\psi\ge 0$, and $\int_{\bR^d}\psi(x)dx=1.$ Set $\phi_n(x)=\psi_n(x-x_0).$ We compute $\bE|v^{(t)}(\phi_n)|^2$ in two ways. First, using \eqref{v-V} and Fubini's theorem, we have
\[
\bE|v^{(t)}(\phi_n)|^2
=
\int_{\bR^d}\int_{\bR^d}\bE[V(t,x)V(t,y)]\phi_n(x)\phi_n(y)dxdy
\to \bE|V(t,x_0)|^2,
\qquad \mbox{as } n\to\infty,
\]
by the Lebesgue differentiation theorem. Secondly, using \eqref{def-vtp} and Fubini's theorem,
\begin{align}
\nonumber
\bE|v^{(t)}(\phi_n)|^2
&=\int_0^t\int_0^t \bE\big[\widetilde{Y}^{(t,\phi_n)}(s)\widetilde{Y}^{(t,\phi_n)}(r)\big]dsdr\\
\nonumber
&=m_2\int_0^t\int_0^t\int_{\bR^d}
(\phi_n*\widetilde{G}_{t-s}*\k)(y)(\phi_n*\widetilde{G}_{t-r}*\k)(y)dydsdr\\
\nonumber
&=m_2\int_0^t\int_0^t\int_{\bR^d}
|\cF\phi_n(\xi)|^2\cF G_{t-s}(\xi)\overline{\cF G_{t-r}(\xi)}\mu(d\xi)dsdr\\
\label{second-way}
&=m_2\int_{\bR^d}
|\cF\phi_n(\xi)|^2\left|\int_0^t \cF G_{t-s}(\xi)ds\right|^2\mu(d\xi).
\end{align}
Since $\cF\phi_n(\xi)\to 1$ as $n\to\infty$, Fatou's lemma yields
\begin{align*}
\int_{\bR^d}\left|\int_0^t \cF G_{t-s}(\xi)ds\right|^2\mu(d\xi)
&\le \liminf_{n\to\infty}
\int_{\bR^d}|\cF\phi_n(\xi)|^2\left|\int_0^t \cF G_{t-s}(\xi)ds\right|^2\mu(d\xi)\\
&=\frac{1}{m_2}\liminf_{n\to\infty}\bE|v^{(t)}(\phi_n)|^2\\
&=\frac{1}{m_2}\bE|V(t,x_0)|^2<\infty.
\end{align*}
Therefore, since $V$ is locally mean-square bounded,
\begin{align*}
\sup_{ t\in [0, T]}\int_{\bR^d}|H_t(\xi)|^2\mu(d\xi)
&\le \frac{1}{m_2}\sup_{t \in [0,T]}\bE|V(t,x_0)|^2<\infty.
\end{align*}

Next, we prove that \eqref{cond2} holds.
By \eqref{v-V},
\[
v^{(t)}(\phi_n)-v^{(s)}(\phi_n)
=
\int_{\bR^d}(V(t,x)-V(s,x))\phi_n(x)dx
\quad \mbox{a.s.}
\]
Hence, by Fubini's theorem,
\begin{align*}
\bE|v^{(t)}(\phi_n)-v^{(s)}(\phi_n)|^2
&=
\int_{\bR^d}\int_{\bR^d}
\bE[(V(t,x)-V(s,x))(V(t,y)-V(s,y))]
\phi_n(x)\phi_n(y)dxdy\\
&\to \bE|V(t,x_0)-V(s,x_0)|^2
\end{align*}
as $n\to\infty$, again by the Lebesgue differentiation theorem.

On the other hand, similarly to \eqref{second-way},
\begin{align*}
\bE|v^{(t)}(\phi_n)-v^{(s)}(\phi_n)|^2
&=
m_2\int_{\bR^d}
|\cF\phi_n(\xi)|^2|H_t(\xi)-H_s(\xi)|^2\mu(d\xi).
\end{align*}
Therefore, by Fatou's lemma,
\begin{align*}
\int_{\bR^d}|H_t(\xi)-H_s(\xi)|^2\mu(d\xi)
&\le \liminf_{n\to\infty}
\int_{\bR^d}
|\cF\phi_n(\xi)|^2|H_t(\xi)-H_s(\xi)|^2\mu(d\xi)\\
&=\frac{1}{m_2}\liminf_{n\to\infty}
\bE|v^{(t)}(\phi_n)-v^{(s)}(\phi_n)|^2\\
&=\frac{1}{m_2}\bE|V(t,x_0)-V(s,x_0)|^2.
\end{align*}
Since $V$ is continuous in mean square, the right-hand side converges to $0$ as $t\to s$.

(Sufficiency). Formally, 
\begin{align*}
\cF\left(\int_0^t (G_{t-s}(x-\cdot)*\k)ds\right)(\xi)
&=\int_0^t \cF\big(G_{t-s}(x-\cdot)*\k\big)(\xi)ds\\
&=\int_0^t \cF G_{t-s}(x-\cdot)(\xi)\cF\k(\xi)ds\\
&=\int_0^t e^{-i\xi\cdot x}\overline{\cF G_{t-s}(\xi)}\cF\k(\xi)ds\\
&=e^{-i\xi\cdot x}\cF\k(\xi)\int_0^t \overline{\cF G_{t-s}(\xi)}ds.
\end{align*}
For $t>0$ and $x\in \bR^d$, define $K_{t,x}\in L^2(\bR^d)$ by
\[
\cF K_{t,x}(\xi)
=
e^{-i\xi\cdot x}\cF\k(\xi)\overline{H_t(\xi)}.
\]
This is well defined since the Fourier transform is an isomorphism on $L^2(\bR^d)$ and
\[
\int_{\bR^d}|\cF K_{t,x}(\xi)|^2d\xi
=
(2\pi)^d\int_{\bR^d}|H_t(\xi)|^2\mu(d\xi)<\infty.
\]
We then set
\[
V(t,x):=L(K_{t,x}).
\]

We first show that $V$ is continuous in $L^2(\Omega)$. By the isometry of $L$,
\begin{align*}
\mathbb E|V(t,x)-V(s,y)|^2
&=m_2\|K_{t,x}-K_{s,y}\|_{L^2(\bR^d)}^2=m_2\int_{\bR^d}
\left|e^{-i\xi\cdot x}{H_t(\xi)}
-
e^{-i\xi\cdot y}{H_s(\xi)}
\right|^2\mu(d\xi).
\end{align*}
Hence
\begin{align*}
\mathbb E|V(t,x)-V(s,y)|^2
&\le 2m_2\int_{\bR^d}|H_t(\xi)-H_s(\xi)|^2\mu(d\xi)\\
&\quad +2m_2\int_{\bR^d}|e^{-i\xi\cdot x}-e^{-i\xi\cdot y}|^2|H_s(\xi)|^2\mu(d\xi).
\end{align*}
The first term converges to $0$ as $t\to s$ by the assumed continuity of $t\mapsto H_t$ in
$L^2(\mu)$. The second converges to $0$ as $x\to y$ by dominated convergence, since
\[
|e^{-i\xi\cdot x}-e^{-i\xi\cdot y}|^2|H_s(\xi)|^2\le 4|H_s(\xi)|^2
\]
and $|H_s|^2\in L^1(\mu)$. Therefore $V$ is continuous in $L^2(\Omega)$. In particular, $V$ is continuous in probability,
and hence admits a jointly measurable modification, which we denote again by $V$.
Moreover, for every $T>0$,
\[
\sup_{0\le t\le T,\ x\in\bR^d}\mathbb E|V(t,x)|^2
=
m_2\sup_{0\le t\le T}\int_{\bR^d}|H_t(\xi)|^2\mu(d\xi)<\infty,
\]
so $V$ is locally mean-square bounded.

Next, fix $t>0$ and $\phi\in \mathcal D(\bR^d)$. We claim that
\[
\mathbb E\left|\int_{\bR^d}V(t,x)\phi(x)dx\right|^2
=
\mathbb E\left(v^{(t)}(\phi)\int_{\bR^d}V(t,x)\phi(x)dx\right)
=
\mathbb E|v^{(t)}(\phi)|^2,
\]
and hence,
$\bE\left|
\int_{\bR^d}V(t,x)\phi(x)dx-v^{(t)}(\phi)
\right|^2=0$, which implies \eqref{v-V}.
Indeed,
\begin{align*}
\mathbb E\left| \int_{\bR^d}V(t,x)\phi(x)dx \right|^2
&= \int_{\bR^d}\int_{\bR^d} \mathbb E\left[V(t,x)V(t,y)\right]\phi(x)\phi(y)dxdy\\
&= \int_{\bR^d}\int_{\bR^d}
m_2\left(\int_{\bR^d}K_{t,x}(z)K_{t,y}(z)dz\right)\phi(x)\phi(y)dxdy\\
&= \frac{m_2}{(2\pi)^d}
\int_{\bR^d}\int_{\bR^d}
\left(\int_{\bR^d}\cF K_{t,x}(\xi)\overline{\cF K_{t,y}(\xi)}d\xi\right)
\phi(x)\phi(y)dxdy\\
&= \frac{m_2}{(2\pi)^d}
\int_{\bR^d}\int_{\bR^d}
\left(\int_{\bR^d}
e^{-i\xi\cdot(x-y)}
|\cF\k(\xi)|^2|H_t(\xi)|^2d\xi\right)
\phi(x)\phi(y)dxdy\\
&= \frac{m_2}{(2\pi)^d}
\int_{\bR^d}
|\cF\phi(\xi)|^2|\cF\k(\xi)|^2|H_t(\xi)|^2d\xi\\
&= m_2\int_{\bR^d}
|\cF\phi(\xi)|^2|H_t(\xi)|^2\mu(d\xi).
\end{align*}

Similarly,
\begin{align*}
\mathbb E\left[\left(\int_{\bR^d}V(t,x)\phi(x)dx\right)v^{(t)}(\phi)\right]
&= \int_{\bR^d}\phi(x)\mathbb E\big[V(t,x)v^{(t)}(\phi)\big]dx\\
&= \int_{\bR^d}\phi(x)\int_0^t
\mathbb E\big[V(t,x)\widetilde Y^{(t,\phi)}(s)\big]dsdx\\
&= m_2\int_{\bR^d}\phi(x)\int_0^t
\left(\int_{\bR^d}K_{t,x}(z)(\phi*\widetilde G_{t-s}*\k)(z)dz\right)dsdx\\
&= \frac{m_2}{(2\pi)^d}\int_{\bR^d}\phi(x)\int_0^t
\left(\int_{\bR^d}
\cF K_{t,x}(\xi)
\overline{\cF(\phi*\widetilde G_{t-s}*\k)(\xi)}
d\xi\right)dsdx\\
&= \frac{m_2}{(2\pi)^d}\int_{\bR^d}\phi(x)
\left(\int_{\bR^d}
e^{-i\xi\cdot x}\overline{\cF\phi(\xi)}
|\cF\k(\xi)|^2|H_t(\xi)|^2
d\xi\right)dx\\
&= \frac{m_2}{(2\pi)^d}\int_{\bR^d}
|\cF\phi(\xi)|^2|\cF\k(\xi)|^2|H_t(\xi)|^2d\xi\\
&= m_2\int_{\bR^d}
|\cF\phi(\xi)|^2|H_t(\xi)|^2\mu(d\xi).
\end{align*}

Finally,
\[
\mathbb E|v^{(t)}(\phi)|^2
=
m_2\int_{\bR^d}
|\cF\phi(\xi)|^2|H_t(\xi)|^2\mu(d\xi).
\]

Suppose now that the random-field solution $v$ of \eqref{lin-eq} exists and  $v(t,x)=L(J_{t,x})$ where
$J_{t,x}(y)=\int_0^t (G_{t-s}(x-\cdot)*\k)(y)ds$.
Then $J_{t,x}$ coincide with $K_{t,x}$ in $L^2(\bR^d)$, since they have the same Fourier transforms. Therefore
$v(t,x)=L(J_{t,x})=L(K_{t,x})=V(t,x)$ in $L^2(\Omega)$.
So $V$ is a modification of the random-field solution $v$.
\end{proof}

\begin{corollary}
Suppose that  (D) holds, and $\cL$ is the heat operator or the wave operator in dimension $d\geq 1$. Let $\{v^{(t)};t\geq 0\}$ and $\{v(t,x);t\geq 0,x\in \bR^d\}$ be the distribution solution, respectively the random field solution, of equation \eqref{lin-eq}. Then relation \eqref{v-V} holds for any $t>0$ and $\phi \in \cD(\bR^d)$, and $V$ is a modification of $v$.
\end{corollary}

\begin{proof} Hypothesis F clearly holds. Conditions \eqref{cond1} and \eqref{cond2} hold under (D),
since by Cauchy-Schwarz inequality, 
\[
|H_t(\xi)-H_s(\xi)|^2 \leq (t -s) \int_{s}^t |\cF G_r(\xi)|^2 dr \quad \mbox{for all} \ 0\leq s <t \leq T.
\]
\end{proof}

\section{Equations with multiplicative noise}
\label{section-mult}

In this section, we study equations with multiplicative noise, of the form:
\begin{equation}
\label{mult-eq}
\cL u(t,x)=u(t,x)\dot{X}(x) \quad t>0,x\in \bR^d,
\end{equation}
with initial condition 1. When $\cL$ is the heat operator, equation \eqref{mult-eq} is called the {\em parabolic Anderson model}, and when $\cL$ is the wave operator, equation \eqref{mult-eq} is called the {\em hyperbolic Anderson model}.  

Intuitively, a solution $u$ to \eqref{mult-eq} satisfies:
\begin{align*}
u(t,x)&=1+\int_0^t \int_{\bR^d}G_{t-s}(x-y)u(s,y)X(dy)ds\\
&=1+\int_0^t \int_{\bR^d} \big(G_{t-s}(x-\cdot)u(s,\cdot)*\k\big)(y) L(dy)ds\\
&=1+\int_0^t \int_{\bR^d \times \bR_0} \big(G_{t-s}(x-\cdot)u(s,\cdot)*\k \big)(y) z \widehat{N}(dy,dz)ds.
\end{align*}
Two problems arise immediately from this formulation: \\
(i) there is no stochastic integral with respect to $X$, $L$ or $\widehat{N}$ for {\em random} integrands since this noise does not have a time component, and there is no concept of ``predictability''; \\
(ii) $G_{t-s}(x-\cdot)u(s,\cdot)$ is a product between a distribution and a function. 

Therefore, the only tool that we have for studying equation \eqref{mult-eq} will be based on the Poisson-chaos expansion.

\subsection{Malliavin calculus preliminaries}

In this section, we include some basic material about Malliavin calculus with respect to the compensated Poisson random measure $\widehat{N}$.

We recall that $\fH=L^2(Z,\cZ,\fm)$, where $(Z,\cZ,\fm)$ is given by \eqref{def-Z}.
Then $\cH^{\otimes n}=L^2(Z^n,\cZ^{n},\fm^n)$. We let $\cH^{\odot n}$ be the set of symmetric functions in $\cH^{\otimes n}$.

\medskip

$\bullet$
Any random variable $F \in L^2(\Omega)$ which is $\cF^N$-measurable has the {\em Poisson-chaos expansion}:
\begin{equation}
\label{Poisson-chaos}
F=\bE(F)+\sum_{n\geq 1}I_n(f_n), \quad \mbox{for some $f_n \in \cH^{\odot n}$},
\end{equation}
where $I_n$ is the multiple integral with respect to $\widehat{N}$, and the series is orthogonal in $L^2(\Omega)$. 
For any $f \in \fH^{\otimes n}$, $I_n(f)=I_n(\widetilde{f})$,
\[
\bE[I_n(f)]=0 \quad \mbox{and} \quad \bE|I_n(f)|^2=n! \|\widetilde{f}\|_{\fH^{\otimes n}}^2,
\]
where $\widetilde{f}$ is the symmetrization of $f$, defined by:
\[
\widetilde{f}(\xi_1,\ldots,\xi_n)=\frac{1}{n!}\sum_{\rho \in S_n} f(\xi_{\rho(1)},\ldots,\xi_{\rho(n)}), \quad \mbox{for $\xi_1,\ldots,\xi_n \in {\bf Z}$},
\]
and $S_n$ is the set of permutations of $1,\ldots,n$. 

\medskip
$\bullet$
For any random variable $F \in L^2(\Omega)$ with chaos expansion \eqref{Poisson-chaos}, we define the {\em Malliavin derivative} of $F$ (with respect to $\widehat{N}$) by:
\[
D_{\xi}F=\sum_{n\geq 1}nI_{n-1}\big(f_n(\cdot,\xi)\big), \quad \mbox{for all} \quad \xi \in {\bf Z},
\]
provided that 
\[
\bE\|DF\|_{\fH}^2=\sum_{n\geq 1}nn! \|\widetilde{f}_n\|_{\fH^{\otimes n}}^2<\infty.
\] In this case, we write $F \in {\rm Dom}(D)$.

\medskip

$\bullet$
We denote by $\delta:{\rm Dom}(\delta)\to L^2(\Omega) $ the adjoint of $D$, where ${\rm Dom}(\delta)$ is the set of $V \in L^2(\Omega;\fH)$ for which there exists a constant $C=C_V>0$ depending on $V$, such that 
\[
\big|\bE \langle DF,V \rangle_{\fH}\big| \leq C \|F\|_2 \quad \mbox{for any $F \in {\rm Dom}(D)$}.
\] 
We say that $\delta(V)$ is 
{\em the Skorohod integral} of $V$ with respect to $\widehat{N}$, and we write
\[
\delta(V)=\int_{ \bR^d } \int_{\bR_0}V(x,z)\widehat{N}(\delta x, \delta z).
\]
By duality, for any $V \in {\rm Dom}(\delta)$,
\[
\bE\langle DF,V \rangle_{\fH}=\bE[F \delta(V)] \quad \mbox{for any $F \in {\rm Dom}(D)$}.
\]

\subsection{Existence of solution}

In this section, we study the existence of the solution of equation \eqref{mult-eq}.

\medskip

Assume that a solution $u$ exists. We proceed heuristically to derive the Poisson-chaos expansion of $u(t,x)$.
Formally,
\begin{align*}
u(t,x)&=1+\int_0^t \int_{\bR^d} \int_{\bR^d}G_{t-s_1}(x-y_1')u(s_1,y_1')\k(y_1-y_1')dy_1' L(dy_1)ds_1.
\end{align*}
We write
\[
u(s_1,y_1')=1+\int_0^{s_1} \int_{\bR^d}\int_{\bR^d} G_{s_1-s_2}(y_1'-y_2')u(s_2,y_2')\k(y_2-y_2')dy_2'L(dy_2)ds_2
\]
and we insert this into the equation above. We obtain:
\begin{align*}
& u(t,x)=1+\int_0^t \int_{\bR^d} \int_{\bR^d}G_{t-s_1}(x-y_1')\k(y_1-y_1')dy_1' L(dy_1)ds_1+\\
&\int_{0<s_2<s_1<t} \int_{(\bR^d)^4} G_{t-s_1}(x-y_1')G_{s_1-s_2}(y_1'-y_2') u(s_2,y_2') \k(y_1-y_1') \k(y_2-y_2') dy_1' dy_2'L(dy_1) L(dy_2)ds_1 ds_2.
\end{align*}
We replace $u(s_s,y_2')$, and we continue in this manner.
Denoting (still formally, since $G_t$ may be a distribution),
\[
f_n(t_1,x_1,\ldots,t_n,x_n,t,x)=G_{t-t_n}(x-x_n)\ldots G_{t_2-t_1}(x_2-x_1) 1_{\{0<t_1<\ldots<t_n<t\}},
\]
we arrive at:
\begin{align*}
u(t,x)&=1+\int_0^t \int_{\bR^d}\big(f_1(s_1,\cdot,t,x)*\k\big)(y_1)L(dy_1)ds_1+\\
& \quad \int_{0<s_2<s_1<t}  \int_{(\bR^d)^2}\big(f_2(s_2,\cdot,s_1,\cdot,t,x)*\k^{\otimes 2}\big)(y_2,y_1) L(dy_2)L(dy_1) ds_2 ds_1+\ldots
\end{align*}
and then we pass to the multiple integrals with respect to $\widehat{N}$.

\medskip

This heuristic argument motivates the following rigorous definition.

We denote $\pmb{t}=(t_1,\ldots,t_n)$ and  $\pmb{x}=(x_1,\ldots,x_n)$, and we define
\[
T_n(t)=\{\pmb{t} \in [0,t]^n; 0<t_1<\ldots<t_n<t\} \quad \mbox{and} \quad \k^{\otimes n}(\pmb{x})=\prod_{j=1}^n\k(x_j).
\] 
For any $(t,x) \in \bR_{+} \times \bR^d$ fixed, let
$f_n(t_1,\cdot,\ldots,t_n,\cdot,t,x)$ be the distribution in $\cS'(\bR^{nd})$ whose Fourier transform is:
\[
\cF f_n(t_1,\cdot,\ldots,t_n,\cdot,t,x)(\pmb{\xi})=e^{-i(\xi_1+\ldots+\xi_n)\cdot x} \prod_{j=1}^{n}\overline{\cF G_{t_{j+1}-t_j}(\xi_1+\ldots+\xi_j)},
\]
with $\pmb{\xi}=(\xi_1,\ldots,\xi_n) \in (\bR^d)^n$, and the convention
 $t_{n+1}=t$, $x_{n+1}=x$.

 In the case of the heat equation or wave equation with $d\leq 2$, $G_t$ is a function, and $f_n(t_1,\cdot,\ldots,t_n,\cdot,t,x)$ is a function given by:
\begin{equation}
\label{def-fn-func}
f_n(t_1,x_1,\ldots,t_n,x_n,t,x)=\prod_{j=1}^{n}G_{t_{j+1}-t_j}(x_{j+1}-x_j), \quad \mbox{for  $\pmb{x}\in (\bR^d)^n$},
\end{equation}
and 
\[
\big(f_n(t_1,\cdot,\ldots,t_n,\cdot,t,x)* \k^{\otimes n}\big)(\pmb{x}) =\int_{(\bR^d)^n}\prod_{j=1}^{n}G_{t_{j+1}-t_j}(y_{j+1}-y_j) \k(x_j-y_j) d\pmb{y}.
\]
 \medskip
 
We will need the following stronger forms of Hypotheses A, B, D and E.

\medskip

\noindent {\bf Hypothesis A'.} For any $(t,x)\in \bR_{+} \times  \bR^d$, $n\geq 1$ and $\pmb{t} \in T_n(t)$,
\[
f_n(t_1,\cdot,\ldots,t_n,\cdot,t,x) \in \cO_{C}'(\bR^{nd}).
\]

\medskip

\noindent {\bf Hypothesis B'.} For any $t>0$,
$K(t):= \sup_{\eta \in \bR^d} \int_{\bR^d}|\cF G_t(\xi+\eta)|^2 \mu(d\xi)<\infty$.\\

\medskip

\noindent {\bf Hypothesis D'.} The map $(t,\xi)\mapsto \cF G_t(\xi)$ is measurable on $\bR_{+} \times \bR^d$, and 
\begin{equation}
\label{def-AT}
A_T:=\int_0^T \sup_{\eta \in \bR^d} \int_{\bR^d}|\cF G_t(\xi+\eta)|^2 \mu(d\xi)dt<\infty \quad \mbox{for any $T>0$}.
\end{equation}

\medskip

\noindent {\bf Hypothesis E'.}
For any $(t,x)\in \bR_{+} \times \bR^d$ and $n\geq 1$, either one of the following conditions holds:\\
(i) the map $({\pmb{t}},\pmb{x}) \mapsto \big(f_n(t_1,\cdot,\ldots,t_n,\cdot,t,x)* \k^{\otimes n}\big)(\pmb{x}) $ is measurable on $T_n(t) \times \bR^{nd}$; \\
(ii) there exists a measurable function $g_n:T_n(t) \times \bR^{nd} \to \bR$ depending on $(t,x,n)$, such that 
\begin{equation}
\label{G*k-meas}
g_n({\pmb{t}},\cdot)= f_n(t_1,\cdot,\ldots,t_n,\cdot,t,x)* \k^{\otimes n} \ \mbox{in $L^2(\bR^{nd})$, \  for almost all $\pmb{t} \in T_n(t)$}.
\end{equation}
We identify $f_n(t_1,\cdot,\ldots,t_n,\cdot,t,x)* \k^{\otimes n}$ with $g_n(\pmb{t},\cdot)$, and we write $f_n(t_1,\cdot,\ldots,t_n,\cdot,t,x)*\k^{\otimes n}$ instead of $g_n(\pmb{t},\cdot)$.

\medskip

For any $(x_1,z_1), \ldots,(x_n,z_n) \in {\bf Z}$, we denote
\[
f_n^{*}(x_1,z_1,\ldots,x_n,z_n,x;t):=\prod_{i=1}^{n}z_i \int_{T_n(t)} \big(f_n(t_1,\cdot,\ldots,t_n,\cdot,t,x)*\k^{\otimes n}\big)(\pmb{x}) d\pmb{t},
\]

By Hypothesis E' and Fubini theorem, the integral above is well-defined (but may be infinite), and  $f_n^*(\cdot, x;t)$ is measurable on ${\bf Z}^n$.

\medskip

We introduce now the definition of the solution to equation \eqref{mult-eq}.

\begin{definition}
{\rm The process $\{u(t,x);t\geq 0,x\in \bR^d\}$ with $\bE|u(t,x)|^2<\infty$ for all $(t,x) \in \bR_{+} \times \bR^d$ and the Poisson-chaos expansion:
\begin{equation}
\label{def-u-m}
u(t,x)=1+\sum_{n\geq 1}I_n\big(f_n^{*}(\cdot,x;t)\big),
\end{equation}
is the {\bf (Skorohod) solution} of equation \eqref{mult-eq}, provided that $f_n^{*}(\cdot,x;t)\in \cH^{\otimes n}$ for any $(t,x)\in \bR_{+}\times \bR^d$ and $n\geq 1$, and the series converges in $L^2(\Omega)$.
}
\end{definition}

We denote
\begin{equation}
\label{def-Jn}
J_n(t)=\int_{T_n(t)}  \int_{(\bR^d)^n}  \prod_{j=1}^{n}|\cF G_{t_{j+1}-t_j}(\xi_1+\ldots+\xi_j)|^2 \mu(d\xi_1) \ldots \mu(d\xi_n) d\pmb{t}.
\end{equation}

\begin{theorem}
Suppose that Hypotheses A', B', D' and E' hold. Then $f_n^{*}(\cdot,x;t)\in \cH^{\otimes n}$ for any $n\geq 1$ and $(t,x)\in \bR_{+}\times \bR^d$. If in addition,
\begin{equation}
\label{sum-Jn}
\sum_{n\geq 1}m_2^n t^n J_n(t)<\infty \quad \mbox{for any $t>0$},
\end{equation}
then equation \eqref{mult-eq} has a unique solution.
\end{theorem}

\begin{proof}
{\em Step 1.} In this step, we show that 
$f_n^{*}(\cdot,x;t)\in \cH^{\otimes n}$ for any $n\geq 1$. 

 By Theorems \ref{GF1} and \ref{GF2}, $f_n(t_1,\cdot,\ldots, t_n,\cdot,t,x)*\k^{\otimes n}\in \cS'(\bR^{nd})$ and
\[
\cF \big(f_n(t_1,\cdot,\ldots,t_n,\cdot, t,x)*\k^{\otimes n}\big)(\pmb{\xi})=\cF f_n(t_1,\cdot,\ldots,t_n,\cdot,t,x)(\pmb{\xi}) \prod_{i=1}^{n}\cF \k(\xi_i),
\]
for any $\pmb{\xi}\in (\bR^d)^n$. By definition \eqref{def-mu} of $\mu$ and Hypothesis B', for all  $\pmb{t} \in T_n(t)$,
\begin{align}
\nonumber
& \int_{(\bR^d)^n} \Big|\cF \big(f_n(t_1,\cdot,\ldots,t_n,\cdot,t,x)*\k^{\otimes n}\big)(\pmb{\xi})\Big|^2 d\pmb{\xi}\\
\nonumber
& \quad =(2\pi)^{nd}\int_{(\bR^d)^n} \prod_{j=1}^{n}|\cF G_{t_{j+1}-t_j}(\xi_1+\ldots+\xi_j)|^2 \mu(d\xi_1)\ldots \mu(d\xi_n)\\
\label{Jn1}
& \quad \leq (2\pi)^{nd} K(t-t_n) K(t_n-t_{n-1}) \ldots K(t_2-t_1)<\infty,
\end{align}
and hence,
$\cF \big(f_n(t_1,\cdot,\ldots,t_n,\cdot,t,x)*\k^{\otimes n}\big) \in L^2(\bR^{nd})$. By Lemma A.1 of \cite{BJ26}, it follows that $f_n(t_1,\cdot,\ldots,t_n,\cdot,t,x)*\k^{\otimes n} \in L^2(\bR^{nd})$ for all $\pmb{t} \in T_n(t)$. By Cauchy-Schwarz inequality and Plancherel's theorem,
\begin{align*}
& \|f_n^*(\cdot,x;t)\|_{\cH^{\otimes n}}^2  = \int_{(\bR^d \times \bR_0)^n} \big|f_n^*(x_1,z_1,\ldots,x_n,z_n,x;t)\big|^2 d\pmb{x} \nu(dz_1) \ldots \nu(dz_n)\\
& \quad  =m_2^n\int_{(\bR^{d})^n} \left(\int_{T_n(t)} \big(f_n(t_1,\cdot,\ldots,t_n,\cdot,t,x)*\k^{\otimes n}\big)(\pmb{x}) d\pmb{t}\right)^2 d\pmb{x}\\
& \quad \leq  \frac{m_2^n t^n}{n!} \int_{T_n(t)}  \int_{(\bR^{d})^n}\big| \big(f_n(t_1,\cdot,\ldots,t_n,\cdot,t,x)*\k^{\otimes n}\big) (\pmb{x}) \big|^2 d\pmb{x} d\pmb{t} \\
&  \quad   = \frac{m_2^n t^n}{n!} \frac{1}{(2\pi)^{nd}}\int_{T_n(t)} \int_{(\bR^{d})^n} \Big|\cF \big(f_n(t_1,\cdot,\ldots,t_n,\cdot,t,x)(\pmb{\xi})\Big|^2 \prod_{j=1}^{n}|\cF\k(\xi_j)|^2 d\pmb{\xi}d\pmb{t} \\
& \quad =\frac{ m_2^n  t^n}{n!} J_n(t).
\end{align*}

Note that by \eqref{Jn1},
\[
J_n(t) \leq  \int_{T_n(t)} K(t-t_n) K(t_n-t_{n-1}) \ldots K(t_2-t_1)  d\pmb{t}  \leq
 A_t^n<\infty.
\]
This proves that $f_n^*(\cdot,x;t) \in \cH^{\otimes n}$.

\medskip

{\em Step 2.} In this step, we prove that the series \eqref{def-u-m} converges in $L^2(\Omega)$, for any $(t,x)\in \bR_{+} \times \bR^d$.
Using the fact that $\|\widetilde{f}\|_{\cH^{\otimes n}} \leq \|f\|_{\cH^{\otimes n}}$, we see that
\[
\sum_{n\geq 1}\bE\big|I_n\big(f_n^*(\cdot,x;t)\big)\big|^2 =\sum_{n\geq 1}n! \|\widetilde{f}_n^*(\cdot,x;t)\|_{\cH^{\otimes n}}^2 \leq \sum_{n\geq 1}n! \|f_n^*(\cdot,x;t)\|_{\cH^{\otimes n}}^2 \leq \sum_{n\geq 1} m_2^n t^n J_n(t).
\]

\end{proof}

Next, we examine the existence of the solution for the heat and wave equations. We start by establishing for validity of Hypotheses A', B', D' and E'. 

\begin{lemma}
Hypothesis A' holds for the fundamental solutions of the heat and wave equations, for any $d\geq 1$.
\end{lemma}

\begin{proof}
In the case of the heat equation, the function $f_n(t_1,\cdot,\ldots,t_n,\cdot,t,x)\in \cS(\bR^{nd})$ and Hypothesis A' follows by Lemma \ref{SO-lem}.

In the case of the wave equation with $d\leq 2$, $f_n(t_1,\cdot,\ldots,t_n,\cdot,t,x)$ is an integrable function with compact support. This function induces a distribution in $\cE'(\bR^{nd}) \subset \cO_{C}'(\bR^{nd})$.

 If $d\geq 3$, $G_t\in \cE'(\bR^d)$ for any $t>0$. Let $(t,x)\in \bR_+\times \bR^d$, $n\geq 1$ and
$\bt=(t_1,\ldots,t_n)\in T_n(t)$. Then the tensor product
\[
G_{t_2-t_1}\otimes \cdots \otimes G_{t_{n+1}-t_n}\in \cE'(\bR^{nd}).
\]
By definition, $f_n(t_1,\cdot,\ldots,t_n,\cdot,t,x)$ is obtained from the distribution 
$ G_{t_2-t_1}\otimes \cdots \otimes G_{t_{n+1}-t_n}$ by the change of variables
\[
(x_1,\ldots,x_n)\mapsto (x_2-x_1,\ldots,x_n-x_{n-1},x-x_n).
\]
It follows that $f_n(t_1,\cdot,\ldots,t_n,\cdot,t,x)\in \cE'(\bR^{nd})\subset \mathcal O_C'(\bR^{nd})$.

\end{proof}

\begin{lemma}
\label{lem-hypB'}
Hypotheses B' and D' hold for the fundamental solutions of the heat and wave equations with $d\geq 1$, provided that (D) holds.
\end{lemma}

\begin{proof}
In the case of the heat equation, it is known that for any $t>0$,
\begin{equation}
\label{max-heat}
K(t)=\sup_{\eta \in \bR^d} \int_{\bR^d} e^{-t|\xi+\eta|^2}\mu(d\xi)= \int_{\bR^d} e^{-t|\xi|^2}\mu(d\xi),
\end{equation}
(see e.g. relation (13) of \cite{BY23}), and  $A_T = \int_0^T \int_{\bR^d} e^{-t|\xi|^2}\mu(d\xi)dt<\infty$, by Lemma \ref{lem-add}. 

Consider next the wave equation.  Note that
\begin{equation}
\label{bound-Gw}
|\cF G_t(\xi)|^2 =\frac{\sin^2(t|\xi|)}{|\xi|^2} \leq D_t \frac{1}{1+|\xi|^2} \quad \mbox{with $D_t=2(t^2 \vee 1)$},
\end{equation}
and
\begin{equation}
\label{max-principle}
\sup_{\eta \in \bR^d}\int_{\bR^d}\frac{1}{1+|\xi+\eta|^2}\mu(d\xi)=\int_{\bR^d}\frac{1}{1+|\xi|^2}\mu(d\xi)=:C_{\mu}.
\end{equation}
Hence, $K(t) \leq D_t C_{\mu}$ for any $t>0$, and
\[
A_T \leq D_T \int_0^T \sup_{\eta \in \bR^d}\int_{\bR^d}\frac{1}{1+|\xi+\eta|^2}\mu(d\xi)dt=T D_T C_{\mu}.
\]
\end{proof}

\begin{lemma}
a) Hypothesis E'.(i) holds for the fundamental solution of the heat equation with $d\geq 1$, and for the fundamental solution of the wave equation with $d\leq 2$.\\
b) Hypothesis E'.(ii) holds for the fundamental solution of the wave equation with $d\geq 3$, provided that (D) holds.
\end{lemma}

\begin{proof} a) The map $(t,x) \mapsto G_t(x)$ is measurable on $\bR_{+} \times \bR^d$. By Fubini's theorem,
\[
(\pmb{t},\pmb{x}) \mapsto \big(f_n(t_1,\cdot,\ldots,t_n,\cdot,t,x)* \k^{\otimes n}\big)(\pmb{x}) \quad \mbox{is measurable on $T_n(t) \times (\bR^d)^n$}.
\]

b) We apply Proposition 1.2.25 of \cite{HNVW} to the spaces $S=T_n(t),T=\bR^{nd},X=\bR$ and the function $F_n:T_n(t)\to L^2(\bR^{nd})$ given by 
\[
F_n(\pmb{t})=f_n(t_1,\cdot,\ldots,t_n,\cdot,t,x)*\k^{\otimes n}.
\] 
By Lemmas \ref{lem-hypB'} and \ref{lem-add}, Hypotheses B' and C are satisfied, since (D) holds.

$F_n$ is well-defined, since by Hypothesis B' and \eqref{Jn1}, for any $\pmb{t}\in T_n(t)$,
\[
\|f_n(t_1,\cdot,\ldots,t_n,\cdot,t,x)*\k^{\otimes n}\|_{L^2(\bR^{nd})}^2=\frac{1}{(2\pi)^{nd}}\int_{\bR^{nd}} \big|\cF \big(f_n(t_1,\cdot,\ldots,t_n,\cdot,t,x)*\k^{\otimes n} \big)(\pmb{\xi}) \big|^2 d\pmb{\xi}<\infty.
\]
Moreover, by Hypothesis C, $F_n$ is continuous, since if $\pmb{t}^{(k)} \to \pmb{t}\in T_n(t)$ as $k\to \infty$, then
\begin{align*}
& \|F_n(\pmb{t}^{(k)})-F_n(\pmb{t})\|_{L^2(\bR^{nd})}^2 =\\
& \quad \int_{(\bR^d)^n} \Big| \prod_{j=1}^{n} \cF G_{t_{j+1}^{(k)}-t_j^{(k)}}(\xi_1+\ldots+\xi_j)- \prod_{j=1}^{n} \cF G_{t_{j+1}-t_j}(\xi_1+\ldots+\xi_j)\Big|^2 \mu(d\xi_1) \ldots \mu(d\xi_n) \to 0,
\end{align*}
as $k \to \infty$, by the dominated convergence theorem. To justify the application of this theorem, we use the following inequality, which can be proved by induction: for any $z_1,w_1,\ldots,z_n,w_n \in \bC$,
\[
\Big|\prod_{j=1}^n z_j - \prod_{j=1}^n w_j \Big| \leq \sum_{j=1}^n |z_1 \ldots z_{j-1}||z_j-w_j| |w_{j+1}\ldots w_n|.
\]
Hence,
\begin{align*}
& \Big| \prod_{j=1}^{n} \cF G_{t_{j+1}^{(k)}-t_j^{(k)}}\big(\sum_{\ell=1}^j\xi_{\ell}\big)- \prod_{j=1}^{n} \cF G_{t_{j+1}-t_j}\big(\sum_{\ell=1}^j\xi_{\ell}\big)\Big|^2 \\
& \quad \leq n \sum_{j=1}^{n} \prod_{i=1}^{j-1} \big| \cF G_{t_{i+1}^{(k)}-t_i^{(k)}}\big(\sum_{\ell=1}^i\xi_{\ell}\big)\big|^2 
\big|\cF G_{t_{j+1}^{(k)}-t_j^{(k)}}\big(\sum_{\ell=1}^j\xi_{\ell}\big)-  \cF G_{t_{j+1}-t_j}\big(\sum_{\ell=1}^j\xi_{\ell}\big)\big|^2\\
& \quad \quad \quad
\prod_{i=j+1}^n\big|\cF G_{t_{i+1}-t_i}\big(\sum_{\ell=1}^i\xi_{\ell}\big)\big|^2\\
& \quad \leq 8n D_t^{n-1} \sum_{j=1}^n \prod_{i=1}^n \frac{1}{1+|\xi_1+\ldots+\xi_i|^2},
\end{align*}
which is integrable with respect to $\mu(d\xi_1)\ldots \mu(d\xi_n)$. For this bound, we used estimate \eqref{bound-Gw}, and the fact that for any $t>0$, $h\in [-1,1]$ and $\xi \in \bR^d$,
\[
|\cF G_{t+h}(\xi)-\cF G_t(\xi)|^2 \leq 1_{\{|\xi|\leq 1\}}+\frac{4}{|\xi|^2}1_{\{|\xi|>1\}}\leq \frac{8}{1+|\xi|^2},
\]
which was shown in the proof of Lemma \ref{lem-add}.

Hence, $F_n$ is measurable with respect to the Borel $\sigma$-field of $L^2(\bR^{nd})$.
Since $L^2(\bR^{nd})$ is separable, $F_n$ is strongly measurable (by Corollary 1.1.10 of \cite{HNVW}), and because the Lebesgue measure $\lambda_{nd}$ on $\bR^{nd}$ is $\sigma$-finite, $F_n$ is strongly $\lambda_{nd}$-measurable (by Proposition 1.1.16, {\em ibid}). Moreover,
\[
\int_{T_n(t)}\|F_n(t)\|_{L^2(\bR^{nd})}^2 d\pmb{t}=\int_{T_n(t)} \int_{\bR^{nd}}|\cF f_n(t_1,\cdot,\ldots,t_n,\cdot,t,x)(\pmb{\xi})|^2 \mu(d\xi_1)\ldots \mu(d\xi_n)=J_n(t)<\infty,
\]
and hence, $\int_{T_n(t)} \|F_n(t)\|_{L^2(\bR^{nd})} dt <\infty$. By Proposition 1.2.2, {\em ibid.}, $F_n$ is Bochner integrable. The conclusion follows by Proposition 1.2.25, {\em ibid}.

\end{proof}

\begin{corollary}
Let $d\geq 1$ be arbitrary. 
Suppose that $\cL=\frac{\partial}{\partial t}-\frac{1}{2}\Delta$  or
$\cL=\frac{\partial^2}{\partial t^2}-\Delta$. If condition (D) holds, then equation \eqref{mult-eq} has a unique solution.
\end{corollary}

\begin{proof}
i) We consider first the heat equation. 
By \eqref{max-heat},
\[
J_n(t) \leq \int_{T_n(t)} \int_{(\bR^d)^n} \prod_{j=1}^{n}e^{-(t_{j+1}-t_j)|\xi_j|^2} \mu(d\xi_1)\ldots \mu(d\xi_n) d\pmb{t}=:h_n(t).
\]
Therefore, invoking Lemma 3.8 of \cite{BC18}, we infer that
\[
\sum_{n\geq 1} m_2^n t^n J_n(t)\leq \sum_{n\geq 1}m_2^nt^n h_n(t)<\infty \quad \mbox{for any $t>0$}.
\]

\medskip

ii) We consider next the wave equation. Then \eqref{sum-Jn} holds, since by \eqref{bound-Gw} and \eqref{max-principle}, 
\[
J_n(t) \leq D_t^n \int_{T_n(t)} \int_{(\bR^d)^n} \prod_{j=1}^{n}\frac{1}{1+|\xi_1+\ldots+\xi_j|^2}\mu(d\xi_1)\ldots \mu(d\xi_n)d\pmb{t} \leq D_t^n C_{\mu}^n \frac{t^n}{n!}.
\]
\end{proof}

\begin{remark}[Connection with the Gaussian case]
{\rm
We consider the following equation
\begin{equation}
\label{mult-eqG}
\cL U(t,x)=\sqrt{m_2} U(t,x)\dot{W}(t,x), \quad t>0,x\in \bR^d,
\end{equation}
with initial condition 1, where $\{W(\varphi);\varphi \in \cP_{0,d}(\bR^d)\}$ is an isonormal Gaussian process with covariance:
\[
\bE[W(\varphi)W(\psi)]=\langle \varphi,\psi \rangle_{0}
\]
and $\langle \cdot,\cdot \rangle_{0}$ is given by \eqref{def-prod0}. The solution of equation \eqref{mult-eqG} is defined similarly as above, using the Skorohod integral with respect to $W$.
Moreover, under Hypotheses A', B', C' and E', equation \eqref{mult-eqG} has a unique solution, whose Wiener-chaos expansion is
\[
U(t,x)=1+\sum_{n\geq 1}m_2^{n/2}I_n^{W}\big(f_n(\cdot,x;t)\big),
\]
where $I_n^{W}$ is the multiple Wiener integral of order $n$ with respect to $W$, and
\[
f_n(x_1,\ldots,x_n,x;t)=\int_{T_{n}(t)}f_n(t_1,x_1,\ldots,t_n,x_n,t,x)d\pmb{t}.
\]

By direct calculation,
$\widetilde{f}_n^{*}(x_1,z_1,\ldots,x_n,z_n,x;t)=\prod_{i=1}^{n}z_i (\widetilde{f}_n\big(\cdot,x;t)*\k^{\otimes}\big)(\pmb{x}),
$
and hence, $\|\widetilde{f}_n^{*}(\cdot,x;t)\|_{\cH^{\otimes n}}^2=m_2^n \|\widetilde{f}_n\big(\cdot,x;t)\|_{(\cP_{0,d}(\bR^d))^{\otimes n}}^2$. This means that the solutions $u$ and $U$ have the same second moments:
\[
\bE|u(t,x)|^2=1+\sum_{n\geq 1}n!\|\widetilde{f}_n^{*}(\cdot,x;t)\|_{\cH^{\otimes n}}^2=
1+\sum_{n\geq 1}m_2^n n!\|\widetilde{f}_n\big(\cdot,x;t)\|_{(\cP_{0,d}(\bR^d))^{\otimes n}}^2=\bE|U(t,x)|^2.
\]
}
\end{remark}

\appendix

\section{Distributions with rapid decrease}
\label{app-distr}

In this section, we recall some properties of distributions which are needed in the sequel, following Chapter VII of \cite{schwartz66}.  In this section, all functions and distributions are considered on $\bR^d$, but we drop $\bR^d$ from the notation, e.g. we write $\cD$ instead of $\cD(\bR^d)$.

\medskip

We begin by recalling the definitions of several spaces of functions, and of distributions.  
\begin{itemize}
\item $\cD$ is the set of $C^{\infty}$-functions with compact support, and $\cD'$ is the set of distributions.

\item $\cS$ is the set of  rapidly decreasing $C^{\infty}$-functions, and $\cS'$ is the set of tempered distributions. A $C^{\infty}$-function $\varphi:\bR^d \to \bR$ is
{\em rapidly decreasing} if for any $\alpha=(\alpha_1,\ldots,\alpha_d),\beta=(\beta_1,\ldots,\beta_d) \in \bN^d$, 
\[
\sup_{x\in \bR^d}|x^{\alpha}D^{\beta} \varphi(x)|<\infty,
\]
where $x^{\alpha}=x_1^{\alpha_1}\ldots x_d^{\alpha_d}$ and $D^{\beta}=\frac{\partial^{|\beta|}}{\partial x_1^{\beta_1}\ldots \partial x_d^{\beta_d}}$, with $|\beta|=\sum_{i=1}^d \beta_i$.

\item $\cE$ is the set of $C^{\infty}$-functions, and $\cE'$ is the set of distributions with compact support.

\item $\cO_M$ is the set of $C^{\infty}$-functions $\varphi$ with {\em slow increase} (see p.243 of \cite{schwartz66}), i.e. for any $\alpha=(\alpha_1,\ldots,\alpha_d) \in \bN^d$, there exist $C>0$ and $k \in \bN$ such that
\[
|D^{\alpha} \varphi(x)| \leq C (1+x^2)^k \quad \mbox{for all $x \in \bR^d$}.
\]

\item $\cO_C'$ is the set of distributions $F$ {\em with rapid decrease} (see p.243 of \cite{schwartz66}), i.e.
     \[
     \mbox{$(1+x^2)^{k/2}F$ is a bounded distribution, for any $k \in \bN$.}
      \]
      (The space of bounded distributions is the dual of $\cD_{L^1}=\{\varphi \in \cE; \, D^{\alpha} \varphi \in L^1(\bR^d) \linebreak \mbox{for all $\alpha \in \bN^d$} \}$; see p. 200 of \cite{schwartz66}.)

\end{itemize}

Then
\[
\cD \subset \cS \subset \cO_M \subset \cE \quad \mbox{and} \quad \cE' \subset \cO_C' \subset \cS' \subset \cD'.
\]

If $\varphi \in \cD$ and $F \in \cD'$, then $\varphi F$ and $\varphi * F$ are distributions in $\cD'$ defined by:
\[
\big( \varphi F,\psi\big)=\big(F,\varphi \psi \big) \quad \mbox{and} \quad \big(\varphi *F, \psi\big) =\big(F,\psi * \widetilde{\varphi}\big) 
\]
for all $\psi \in \cD$, 
where $\widetilde{\varphi}(x)=\varphi(-x)$.

\medskip

We recall the following properties of convolutions (see p. 245 of \cite{schwartz66}):\\
a) $\varphi * F \in \cD$ for any $\varphi \in \cD$  and $F \in \cE'$,  where $(\varphi *F)(x):=\big( F, \varphi(x-\cdot)\big)$ for any $x \in \bR^d$;\\
b)  $\varphi * F \in \cS$ for any $\varphi \in \cD$  and $F \in \cO_C'$;\\ 
c)  $\varphi * F \in \cO_M$ for any $\varphi \in \cD$  and $F \in \cS'$.

\medskip

The {\em Fourier transform} of $F \in \cS'$ is a distribution $\cF F \in \cS'$ defined by:
\[
\big( \cF F,\psi\big)=\big(F,\cF \varphi\big) \quad \mbox{for all $\varphi \in \cS$}.
\]

If $G,F \in \cD'$, then the convolution $G*F \in \cD'$ is defined by
\[
\big( G*F,\varphi\big)=\big(F,\varphi* \widetilde{G} \big) \quad \mbox{for all $\varphi \in \cD$},
\]
where $\widetilde{G}\in \cD'$ is defined by $\big(\widetilde{G}, \varphi\big)=\big( G,\widetilde{\varphi}\big)$ for all $\varphi \in \cD$.

\begin{theorem}[Theorem XI, Chapter VII, p.247 of \cite{schwartz66}] 
\label{GF1}
For any $G \in \cO_C'$ and $F \in \cS'$, 
\[
G*F \in \cS'.
\] 

\end{theorem}

\begin{theorem}[Theorem XV, Chapter VII, p.268, ibid.]  
\label{GF2}
a) For any $G \in \cO_C'$, $\cF G \in \cO_M$.

b) For any $G \in \cO_C'$ and $F \in \cS'$,  
\[
\cF(G*F) =\cF G \cdot \cF F.
\]
\end{theorem}

\begin{lemma}
\label{SO-lem}
Any function $\varphi \in \cS$ can be viewed as a distribution in $\cO_{C}'$.
\end{lemma}

\begin{proof}
Let $T_{\varphi}$ be the distribution induced by $\varphi$, given by $(T_{\varphi},\phi)=\int_{\bR^d}\varphi(x) \phi(x) dx$ for all $\phi \in \cD$. We want to prove for any $k\in \bN$,
\[
(1+x^2)^{k/2}T_\varphi\in (D_{L^1})'.
\]
For any $k\in\bN$, $b_k:=(1+x^2)^{k/2}\varphi \in \cS(\bR^d)$, hence it is a bounded function. For any $\psi\in D_{L^1}$, let $T_k(\psi):=\int_{\bR^d} b_k(x)\psi(x)dx$. Note that $T_k$ is well-defined, since
\[
\left|\int_{\bR^d} b_k(x)\psi(x)dx\right| \le \|b_k\|_\infty \|\psi\|_{L^1}.
\]
Moreover, $T_k$ is linear and continuous on $D_{L^1}$, so $T_k\in (D_{L^1})'$. $T_k$ is exactly the distribution associated with $(1+|x|^2)^{k/2}T_{\varphi}$. Hence, $(1+x^2)^{k/2}T_\varphi=T_k \in (D_{L^1})'$.

\end{proof}

\end{document}